\documentclass{amsart}
\usepackage{amssymb}
\usepackage{amsbsy, amsthm, amsmath,amstext, amsopn}
\usepackage{hyperref}
\usepackage{tikz-cd}
\usepackage{multicol}
\usepackage{tikz}
\usetikzlibrary{arrows,backgrounds}
\usetikzlibrary{decorations.markings}
\usetikzlibrary{matrix}
\usetikzlibrary{graphs}

\theoremstyle{definition}
\newtheorem{theorem}{Theorem}[section]
\newtheorem{prop}[theorem]{Proposition}
\newtheorem{lemma}[theorem]{Lemma}

\newtheorem{cor}[theorem]{Corollary}

\theoremstyle{definition}
\newtheorem{conj}[theorem]{Conjecture}
\newtheorem{definition}[theorem]{Definition}

\newtheorem{observation}[theorem]{Observation}

\theoremstyle{remark}
\newtheorem{rmk}[theorem]{Remark}

\newcommand{\F}{\mathbb{F}}
\newcommand{\Z}{\mathbb{Z}}
\newcommand{\Q}{\mathbb{Q}}
\newcommand{\R}{\mathbb{R}}
\newcommand{\Zt}{\mathbb{Z}^t}
\newcommand{\Qt}{\mathbb{Q}^t}
\newcommand{\Rt}{\mathbb{R}^t}
\newcommand{\A}{\mathcal{A}}
\newcommand{\B}{\mathcal{B}}
\newcommand{\X}{\mathcal{X}}
\newcommand{\W}{\mathcal{W}}

\newcommand{\C}{\mathbb{C}}
\newcommand{\vG}{G^\vee}
\newcommand{\cO}{\mathcal{O}}
\newcommand{\cK}{\mathcal{K}}
\newcommand{\Gr}{{\rm Gr}}
\newcommand{\GL}{\textnormal{GL}}
\newcommand{\SL}{\textnormal{SL}}
\newcommand{\PGL}{\textnormal{PGL}}

\newcommand{\val}{\operatorname{val}}
\newcommand{\Spec}{\operatorname{Spec}}
\newcommand{\Conf}{\operatorname{Conf}}

\def\tGr{\widetilde \Gr}

\newcounter{Qcount}
\begin{document}

\stepcounter{Qcount}

\title{Tropicalization of Positive Grassmannians}

\author{Ian Le}
\address{Perimeter Institute for Theoretical Physics\\ Waterloo, ON N2L 2Y5}
\email{ile@perimeterinstitute.ca}

\author{Chris Fraser}
\address{Indiana University-Purdue University Indianapolis\\ Indianapolis, IN 46202}
\email{chfraser@iupui.edu}

\begin{abstract}
We introduce combinatorial objects which are parameterized by the positive part of the tropical Grassmannian $\Gr(k,n)$. Our method is to relate the Grassmannian to configuration spaces of flags. By work of the first author, and of Goncharov and Shen, configuration spaces of flags naturally tropicalize to give configurations of points in the affine building, which we call higher laminations. We use higher laminations to give two dual objects that are parameterized by the positive tropicalization of $\Gr(k,n)$: equivalence classes of higher laminations; or certain restricted subset of higher laminations. This extends results of Speyer and Sturmfels on the tropicalization of $\Gr(2,n)$, and of Speyer and Williams on the tropicalization of $\Gr(3,6)$ and $\Gr(3,7)$. We also analyze the $\X$-variety associated to the Grassmannian, and give an interpretation of its positive tropicalization.
\end{abstract}

\maketitle

\tableofcontents

\section{Introduction}

Given an affine algebraic variety $X$ along with a set of generators for the ring of functions on $X$, we can form the tropicalization of $X$ with respect to these generators. For example, when $X$ is the affine cone over the Grassmannian, it is natural to take the Pl\"ucker coordinates on $X$ as the set of generators.

Many varieties in representation theory also come with positive structures coming from the theory of total positivity \cite{Lu}. A positive structure on an affine variety is defined by giving its semi-ring of positive functions. A positive structure on a variety $X$ allows one to define the set of positive tropical points of $X$, which is generally a subset of the tropical points of $X$. 

In \cite{SS} Speyer and Sturmfels show that the tropicalization of $\Gr(2,n)$ parameterizes metric trees with $n$ leaves. (The leaves are allowed to have negative length, but let us ignore this.) Moreover, they show that the positive part of $\Gr(2,n)$ parameterizes \emph{planar} trees; positivity induces a cyclic ordering of the edges at each vertex. In \cite{SW}, Speyer and Williams explicitly describe the positive tropicalization of $\Gr(3,6)$ and $\Gr(3,7)$. In particular, they give the fan structure on these positive tropical spaces and compare them to some related associahedra. However, they don't give an interpretation of the tropical points that generalizes trees for $\Gr(2,n)$. We wish to give a combinatorial and geometric object that parameterizes the positive tropical points of $\Gr(k,n)$ in general.

We first consider the space $\Conf_n \A$, the space of configurations of $n$ principal flags for $SL_k$. The positive tropical part of $\Conf_n \A$ parameterizes \emph{higher laminations}, defined in \cite{GS}, \cite{Le} (we will give a detailed definition later). We can roughly define higher laminations as certain configurations of points in the affine building for $SL_k$. The affine building is a simplicial complex of dimension $k-1$ equipped with a vector-valued distance function. Tropical functions on $\Conf_n \A$ have an interpretation in terms of the metric geometry of the building. In the case where $k=SL_2$, $\Conf_n \A = \Gr(2,n)$, and the affine building is an infinite tree. Then higher laminations will be metric trees, recovering the results of \cite{SS}.

While this at first seems to be a generalization in a somewhat different direction, we can use this to describe a tropicalization of $\Gr(k,n)$. There is a natural map
$$\Conf_n \A \rightarrow \tGr(k,n).$$
(Here, $\tGr(k,n)$ is the affine cone over $\Gr(k,n)$.) This map turns out to respect positive tropical points, and is surjective on these points. This realizes the positive tropical points of $\tGr(k,n)$ as certain equivalence classes of higher laminations, which we will call \emph{horocycle laminations} (see Definition \ref{horolam}). A horocycle lamination can be thought of as a configuration of horocycles in the affine building. Choosing any point from each of these horocycles gives a configuration in the affine building, and hence a higher lamination. We again will have a metric interpretation of the Pl\"ucker coordinates in terms of metric geometry.

It turns out the horocycle laminations have distinguished higher laminations as representatives; in other words, the horocycles in the building have distinguished points. This gives a different interpretation of the positive tropical points of $\tGr(k,n)$. This interpretation is in a sense dual to the previous one, and arises from some considerations having to do with the Duality Conjectures of \cite{FG1}.

These considerations allow us to deduce a version of the duality conjectures for $\tGr(k,n)$ from the duality conjectures for $\Conf_n \A$. We thus find that certain tropical points of $\tGr(k,n)$ parameterize a basis of functions in $\cO(\tGr(k,n))$. We conjecture a precise relationship between this parameterization and the one given in \cite{RW}. Both these parameterizations have roots in mirror symmetry.

We expect that there is an alternative way to prove these results using the existence of the maximal green sequence/DT transformation for $\tGr(k,n)$ \cite{MS}. The existence of the DT transformation along with the work of \cite{GHKK} allows one to deduce the duality conjectures for the cluster variety associated to  $\tGr(k,n)$. One would further need to analyze the potential that comes from viewing $\tGr(k,n)$ as a partial compactification of the cluster variety in order to get a parameterization of functions on  $\tGr(k,n)$ in terms of polytopes. We do not believe this has been done. In any case, our approach describes these polytopes fairly explicitly.


One of our main tools will be the idea of a \emph{cluster fibration}, which we hope may be of interest in other contexts. In particular, we show that the map from $\Conf_n \A$ to $\tGr(k,n)$ is a cluster fibration.

In Section 2, we give some background on cluster algebras and tropicalization. Section 3 analyzes the relationship between the spaces $\Conf_n \A$ and $\tGr(k,n)$ in the context of positive tropical geometry. In Section 4, we recall the definition and relevant facts about higher laminations. Sections 5 and 6 give the dual realizations of the positive tropical points of $\tGr(k,n)$ as horocycle laminations and their distinguished representatives. The most nontrivial results are in Section~6, which analyzes the map between the tropical points of $\Conf_n \A$ and $\tGr(k,n)$ in terms of convex geometry. Section 7 gives an application to the duality conjectures and relates them to \cite{RW}. Section 8 discusses the cluster $\X$-variety for the Grassmannian. 

\section{Background}
\subsection{Cluster algebras, positive structures and tropicalization}

We will not use much from the theory of cluster algebras in this paper. We mostly use some very soft facts about cluster algebras. We give a quick review of cluster algebras in the appendix. Here, let us summarize the handful of facts that we need.

Cluster algebras are commutative rings that come equipped with a collection of distinguished generators, together with specific groupings of these generators into \emph{clusters}.  Each generator is called a \emph{cluster variable}. Each cluster algebra has a set of variables called \emph{frozen variables} which belong to every cluster. The rest of the variables are the \emph{non-frozen variables}. Each cluster also has an associated skew-symmetric $B$-matrix, which can be encoded by a quiver.

There is a procedure called \emph{mutation} which modifies the cluster in the following way: all the cluster variables in the cluster remain the same except for one, which is mutated, meaning that it is replaced by a rational expression in the original cluster variables. The particular rational expression depends on the $B$-matrix, and a formula can be found in the appendix. The $B$-matrix likewise is modified according to the {\it matrix mutation rule} given in the appendix. The salient point for us is that this rational expression is a \emph{positive} rational expression, meaning that it is subtraction-free. We will see that a cluster structure on a variety induces a positive structure on that variety.

Let us now define the notion of a positive structure on an algebraic variety. We will then explain how the cluster structure on $X$ gives a positive structure on $X$.
Let $Y$ be any algebraic variety. We say that $Y$ has a {\it positive atlas} if it has an atlas of coordinate charts such that all transition functions involve only addition, multiplication and division. Each of the coordinate charts in the positive atlas is called a \emph{positive chart}. When $Y$ has a positive atlas, we say that $Y$ is a {\it positive variety} or that $Y$ has a \emph{positive structure}. Many varieties in representation theory--for example, partial flag varieties, configuration spaces and moduli of local systems--have natural positive structures coming from Lusztig's theory of total positivity. In many cases, this structure comes from a cluster structure on $Y$.

This is the case for the varieties that are of central interest for us: both $\tGr(k,n)$ and $\Conf_n \A$ have cluster structures that induce positive structures on these varieties.

If $Y$ is a positive variety, then a function $f \in \cO(Y)$ is called a {\it positive rational function} if it can be written as a positive rational function of the coordinates in some (and hence in any) positive chart.

If $Y$ is a positive variety we can take points of $Y$ with values in any {\it semifield}, i.e. in any set equipped with operations of addition, multiplication and division, such that these operations satisfy their usual properties (the most important and non-trivial being distributivity). For us, the important examples of semifields will be the positive real numbers $\R_{>0}$; any tropical semifield; and the semifield which interpolates between these two: the field of formal Laurent series over $\R$ with positive leading coefficient, which we denote $\cK_{>0}$. The tropical semifields $\Z^t, \Q^t, \R^t$ are obtained from $\Z, \Q, \R$ by replacing the operations of multiplication, division and addition by the operations of addition, subtraction and taking the maximum, respectively.

The set $Y(\Z^t)$ (respectively, $Y(\Q^t)$, $Y(\R^t)$) is what is commonly called the set of integral (respectively, rational, real) \emph{positive tropical points} of $Y$. It is a subset of the tropical variety associated to $Y$, and can be considered as the positive part of the tropical variety. The reason for this comes from the following alternative description of $Y(\Z^t)$.

Let $x \in Y(\cK_{>0})$. Then there is a corresponding tropical point $x^t$ of the space $Y(\Zt)$. This point $x^t$ is characterized by the property that if $f$ is one of the positive coordinates of a positive chart (or more generally any positive rational function), then $$f(x^t)=-\val f(x).$$
One can check that under any change of positive coordinate charts, because all transition maps are subtraction free, the expression for $-\val f(x)$ transforms tropically. 

Thus we get a map $$-\val:  Y(\cK_{>0}) \rightarrow Y(\Zt).$$ The map is surjective: in any coordinate chart, we may specify the valuations of coordinates of a point of  $Y(\cK_{>0})$ as we wish. Because all transition functions between charts are invertible, specifying the coordinates in one chart is the same as specifying the coordinates in every chart.

Our main goal will to give an interpretation of the space $X(\Z^t)$ where $X=\tGr(k,n)$.

\subsection{The Grassmannian and its cluster structure}
Let $\Gr(k,n)$ be the Grassmannian of $k$-planes in $\C^n$. We will be mostly concerned with the affine cone $X=\tGr(k,n) \rightarrow \Gr(k,n)$ over the Grassmannian. The affine cone $X$ contains a distinguished point known as the {\it cone point}. Away from this point, $X$ is a $\C^*$ bundle over the Grassmannian. We will give some reminders about $X$.

Every point in $X$ (besides the cone point) can be described by a full-rank $k \times n$ matrix $A$, considered modulo left multiplication by $k\times k$ matrices of determinant $1$. The row span of $A$ is a $k$-dimensional subspace of $\C^n$, i.e., a point in $\Gr(k,n)$, and this defines the map $X \to \Gr(k,n)$. In particular, the Grassmannian $\Gr(k,n)$ can be identified with full rank $k \times n$ matrices modulo left multiplication by $\GL_k$. 

We recall the construction of Pl\"ucker coordinates on $X$. We denote by $\binom{[n]}{k}$ the set of all $k$-element subsets of $[n]:=\{1,\dots,n\}$. Let $A$ be a $k \times n$ matrix representing a point in $X$. For $J\in \binom{[n]}{k}$, the Pl\"ucker coordinate $P_J(A)$ is the minor of $A$ consisting of the columns indexed by $J$. The function $P_J(A)$ is a well-defined function on $X$. Furthermore, the map $A \mapsto (P_J(A))$, where $J$ ranges over $\binom{[n]}{k}$, descends to a projective embedding (the {\it Pl\"ucker embedding}) 
$$\Gr(k,n) \hookrightarrow \mathbb{P}^{\binom{n}{k}-1}.$$

When there is no possibility of confusion, we will usually abbreviate the function $P_J$: for example, instead of $P_{1457}$, we will write $1457$. We will also sometimes write $P(1, 4, 5, 7)$ instead of $P_{1457}$.

By considering the columns of $A$, rather than its rows, a point in $X$ can alternatively be viewed as an $n$-tuple of vectors in $k$-dimensional space. More precisely, $X$ is the geometric invariant theory quotient $\SL_k \backslash \backslash (\C^k)^n$ of this space by $\SL_k$ action. Throughout this paper we denote by $\Conf_n V$ the configuration space of $n$-tuples of vectors in $V$, in other words the naive quotient $\SL_k \backslash V^n$.  At the level of functions one has $\cO(X)=\cO(\Conf_n V)$, and for our purposes it will suffice to think of $\Conf_n V$ and $X$ as being one and the same.

The space $X$ has a cluster structure, discovered by Scott \cite{S}. This cluster structure induces a positive structure that has been well-studied \cite{P}. Figure~\ref{Gr48Seed} depicts a cluster in the cluster structure for $\Gr(4,8)$. At each vertex of the quiver we put one of the cluster variables.

\begin{figure}
\begin{tikzpicture}[scale=2]
\begin{scope}[xshift=-1cm]
  \node[] (x00) at (0,0) {$1238$};
  \node[] (x10) at (1,0) {$1237$};
  \node[] (x20) at (2,0) {$1236$};
  \node[] (x30) at (3,0) {$1235$};
  \node[] (x40) at (4,0) {$1234$};

  \node[] (x01) at (0,-1) {$1278$};
  \node[] (x11) at (1,-1) {$1267$};
  \node[] (x21) at (2,-1) {$1256$};
  \node[] (x31) at (3,-1) {$1245$};

  \node[] (x02) at (0,-2) {$1678$};
  \node[] (x12) at (1,-2) {$1567$};
  \node[] (x22) at (2,-2) {$1456$};
  \node[] (x32) at (3,-2) {$1345$};

  \node[] (x03) at (0,-3) {$5678$};
  \node[] (x13) at (1,-3) {$4567$};
  \node[] (x23) at (2,-3) {$3456$};
  \node[] (x33) at (3,-3) {$2345$};

  \draw [->] (x40) to (x30);
  \draw [->] (x30) to (x20);
  \draw [->] (x20) to (x10);
  \draw [->] (x10) to (x00);
  \draw [->] (x31) to (x21);
  \draw [->] (x21) to (x11);
  \draw [->] (x11) to (x01);
  \draw [->] (x32) to (x22);
  \draw [->] (x22) to (x12);
  \draw [->] (x12) to (x02);

  \draw [->] (x02) to (x03);
  \draw [->] (x01) to (x02);
  \draw [->] (x00) to (x01);
  \draw [->] (x12) to (x13);
  \draw [->] (x11) to (x12);
  \draw [->] (x10) to (x11);
  \draw [->] (x22) to (x23);
  \draw [->] (x21) to (x22);
  \draw [->] (x20) to (x21);
  \draw [->] (x32) to (x33);
  \draw [->] (x31) to (x32);
  \draw [->] (x30) to (x31);

  \draw [->] (x03) to (x12);
  \draw [->] (x02) to (x11);
  \draw [->] (x01) to (x10);
  \draw [->] (x13) to (x22);
  \draw [->] (x12) to (x21);
  \draw [->] (x11) to (x20);
  \draw [->] (x23) to (x32);
  \draw [->] (x22) to (x31);
  \draw [->] (x21) to (x30);
  \draw [->] (x33) to (x40);

\end{scope}

\end{tikzpicture}
\caption{\label{Gr48Seed} The cluster structure for $\tGr(4,8)$.}
\end{figure}
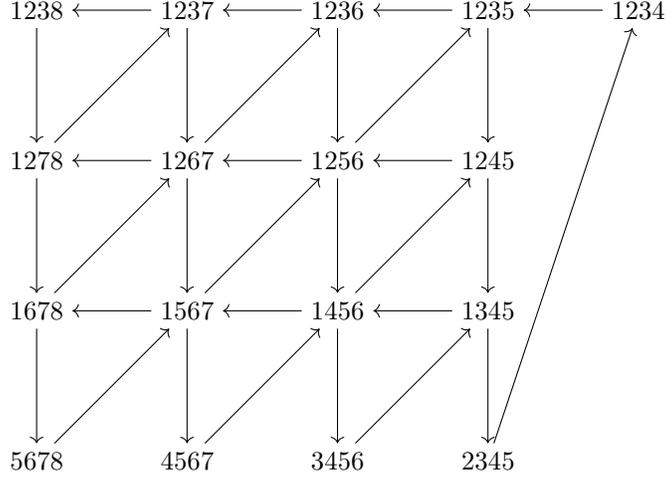

Note that in the above cluster, all the the cluster variables are Pl\"ucker coordinates. Moreover, the frozen variables are the Pl\"ucker coordinates of the form $i, i+1, i+2, i+3$ where indices are taken mod $8$. It should be clear how to generalize this construction to obtain a cluster for $\Gr(k,n)$.

We use the following facts about the cluster structure on the Grassmannian:
\begin{itemize}
\item The frozen cluster variables are the Pl\"ucker coordinates of the form $i, i+2, \dots, i+k-1$ where indices are taken$\mod n$.
\item There is a cluster for $X$ consisting of all Pl\"ucker coordinates of the form $1, 2, 3, \dots, a, i, i+1, \dots, i+b-1$ where $a+b=k$, $a, b \geq 0$, and $i > a$. Note that when $a=0$, we get the frozen Pl\"ucker coordinates $i, i+1, \dots, i+k-1$. (There are many other clusters that consist only of Pl\"ucker coordinates.)
\item There is a \emph{twisted cyclic shift} map $\Conf_n V \xrightarrow{\rho} \Conf_n V$ that sends
$$(v_1, v_2, \dots, v_n) \rightarrow (v_2, \dots, v_n, v_1)$$
when $k$ is odd, and when $k$ is even sends
$$(v_1, v_2, \dots, v_n) \rightarrow (v_2, \dots, v_n, -v_1).$$
This is an {\it automorphism} of the cluster structure: pulling back the functions in any cluster by $\rho$ produces another set of functions forming a cluster.
\end{itemize}

\subsection{Configuration spaces of flags and their cluster structure}\label{flags}

Let $G=\SL_k=\SL(V)$, where as before $V=\C^k$. Moreover, let $B$ be a Borel subgroup of $G$, and let $U := [B,B]$ be a maximal unipotent subgroup of $G$. Then $\A := G/U$ is the {\it principal flag variety}. A point in $\A$ corresponds to a \emph{principal affine flag}. In concrete terms, a principal affine flag can be described by giving an ordered set of $k$ basis vectors $v_1, \dots, v_k$. A set of $k$ vectors determines a flag by considering the $i$-dimensional subspaces spanned by $v_1, \dots , v_i$ for $i \leq k$. These $i$-dimensional subspaces have natural volume forms $$v_1 \wedge \dots \wedge v_i$$ for $k=1, 2, \dots, k-1$. Because we are considering $\SL_k$-flags, we will require that $$v_1 \wedge \dots \wedge v_k$$ is the standard volume form. Two sets of basis vectors determine the same principal affine flag if they give the same $i$-forms $v_1 \wedge \dots \wedge v_i$ for $i \leq k$. We will sometimes call these objects \emph{principal flags}, or merely flags, for short.

We define the configuration space of $n$ principal flags $\Conf_n \A$ as the quotient of $\A^n$ by the diagonal action of $G$:
$$\Conf_n \A := G \backslash \A^n.$$

We now describe the cluster structure on the space $\Conf_n \A$. First we define some functions on $\Conf_3 \A$. 

Suppose that we have a point in $\Conf_3 \A$ given by three flags $A, B, C$, which are represented by $u_1, \dots, u_n$, $v_1, \dots, v_n$ and $w_1, \dots, w_n$ respectively. Fock and Goncharov define a canonical function $f_{i_1i_2i_3}$ of this triple of flags for every triple of non-negative integers $i_1, i_2, i_3$ such that $i_1+i_2+i_3=k$ and $i_1, i_2, i_3 < k$. It is defined by 
$$f_{i_1i_2i_3}(A, B, C)=\det(u_1, u_2, \dots, u_{i_1}, v_1, v_2, \dots v_{i_2}, w_1, w_2, \dots, w_{i_3}),$$ 
and it is $G$-invariant by definition. We will sometimes abbreviate the function $f_{i_1i_2i_3}(A, B, C)$ by $A^{i_1}B^{i_2}C^{i_3}$.

Let us let the flags $A, B, C$ label the vertices of a triangle. Each of the above functions is associated to either an edge of this triangle or its interior. When one of $i_1, i_2, i_3$ is $0$, the function $f_{i_1i_2i_3}$ only depends on two of the flags. We can call such functions {\em edge} functions. Moreover, the functions which depend only on, say, $A$ and $B$, we associate to the edge $AB$. We will call the remaining functions {\em face} functions, and associate them to the interior of the triangle. 
Now we can describe the cluster structure on $\Conf_3 \A$. We do this by the example. Below is a picture of the functions and the quiver for $\Conf_3 \A$ when $G=SL_5$. We can imagine the flags $A, B, C$ at the three vertices (top, bottom and right) of the triangle containing the quiver, so that the edge functions are on the edges of the triangle, and the face functions are in the interior of the triangle: 

\begin{figure}
\begin{tikzpicture}[scale=2]

  \node (x410) at (-2,0) {$A^4B^1$};
  \node (x320) at (-2,-1) {$A^3B^2$};
  \node (x230) at (-2,-2) {$A^2B^3$};
  \node (x140) at (-2,-3) {$A^1B^4$};

  \node (x401) at (-1,0.5) {$A^4C^1$};
  \node (x311) at (-1,-0.5) {$A^3B^1C^1$};
  \node (x221) at (-1,-1.5) {$A^2B^2C^1$};
  \node (x131) at (-1,-2.5) {$A^1B^3C^1$};
  \node (x041) at (-1,-3.5) {$B^4C^1$};

  \node (x302) at (0,0) {$A^3C^2$};
  \node (x212) at (0,-1) {$A^2B^1C^2$};
  \node (x122) at (0,-2) {$A^1B^2C^2$};
  \node (x032) at (0,-3) {$B^3C^2$};

  \node (x203) at (1,-0.5) {$A^2C^3$};
  \node (x113) at (1,-1.5) {$A^1B^1C^3$};
  \node (x023) at (1,-2.5) {$B^2C^3$};

  \node (x104) at (2,-1) {$A^1C^4$};
  \node (x014) at (2,-2) {$B^1C^4$};


  \draw [->] (x014) to (x104);
  \draw [->] (x023) to (x113);
  \draw [->] (x113) to (x203);
  \draw [->] (x032) to (x122);
  \draw [->] (x122) to (x212);
  \draw [->] (x212) to (x302);
  \draw [->] (x041) to (x131);
  \draw [->] (x131) to (x221);
  \draw [->] (x221) to (x311);
  \draw [->] (x311) to (x401);
  \draw [->, dashed] (x140) to (x230);
  \draw [->, dashed] (x230) to (x320);
  \draw [->, dashed] (x320) to (x410);

  \draw [->] (x401) to (x410);
  \draw [->] (x302) to (x311);
  \draw [->] (x311) to (x320);
  \draw [->] (x203) to (x212);
  \draw [->] (x212) to (x221);
  \draw [->] (x221) to (x230);
  \draw [->] (x104) to (x113);
  \draw [->] (x113) to (x122);
  \draw [->] (x122) to (x131);
  \draw [->] (x131) to (x140);
  \draw [->, dashed] (x014) to (x023);
  \draw [->, dashed] (x023) to (x032);
  \draw [->, dashed] (x032) to (x041);

  \draw [->] (x140) to (x041);
  \draw [->] (x230) to (x131);
  \draw [->] (x131) to (x032);
  \draw [->] (x320) to (x221);
  \draw [->] (x221) to (x122);
  \draw [->] (x122) to (x023);
  \draw [->] (x410) to (x311);
  \draw [->] (x311) to (x212);
  \draw [->] (x212) to (x113);
  \draw [->] (x113) to (x014);
  \draw [->, dashed] (x401) to (x302);
  \draw [->, dashed] (x302) to (x203);
  \draw [->, dashed] (x203) to (x104);

\end{tikzpicture}
\caption{\label{fig:Sl5seed} A cluster and quiver for $\Conf_3 \mathcal{A}$ when $G = \SL_5$.}
\end{figure}
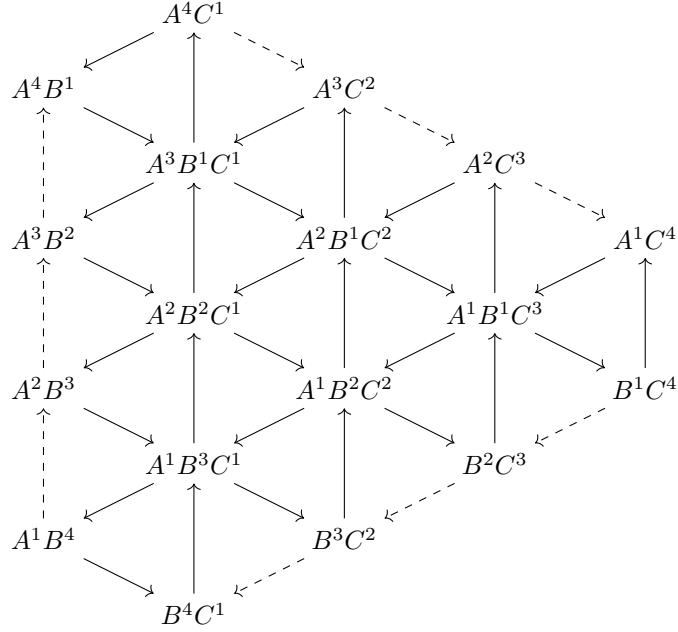

This now easily extends to $\Conf_n \A$. Suppose we have a configuration of $n$ flags. The $n$ flags come as an ordered set of flags, but what is most important is their cyclic order. We can imagine the flags sitting at the vertices of an $n$-gon. To get the cluster structure, triangulate the $m$-gon. The entire quiver is glued out of the quivers for each of the $n-2$ triangles in the triangulation. Each triangle in the triangulation has three associated flags. We place the quiver for $\Conf_3 \A$ and the corresponding functions inside each triangle. 

When an edge belongs to two adjacent triangles, the functions along this edge in each triangle are the same, and when we glue the triangles together, we identify the edge functions. The rule for forming the quiver is that dotted arrows count as a half-arrow, so that two dotted arrows in the same direction glue to give us a solid arrow, whereas two dotted arrows in the opposite direction cancel

The frozen variables will be the functions associated to the edges at the boundary of the $n$-gon. This completely describes the cluster structure on $\Conf_n \A$.

The main fact we will need is the following:

\begin{theorem} (\cite{FG1}) For any triangulation, the set of all edge and face functions for that triangulation forms a cluster. Different triangulations yield different clusters that can be related by a sequence of mutations. In particular, the these clusters give us positive coordinate charts whose transition functions are given by positive rational transformations (transformations involving only addition, multiplication and division).
\end{theorem}

There is also a twisted cyclic shift map for $\Conf_n \A$:
\begin{align*}
\tau: \Conf_n \A &\rightarrow \Conf_n \A \\
\tau(F_1, F_2, \dots F_n) &= (F_2, \dots, F_n, s_GF_1)
\end{align*}
where $s_G=I$, the identity map, when $k$ is odd and $s_G=-I$ when $k$ is even. Again, it follows from \cite{FG1} that the twisted cyclic shift of any cluster is also a cluster.

\section{Natural maps between flag and vector configuration spaces}\label{natural}

In this section, we discuss various relationships between the spaces $\Conf_m \A$ and $\Conf_n V$ and their cluster/positive structures.

The most obvious relationship is that there is a forgetful map

$$\pi: \Conf_n \A \rightarrow \Conf_n V$$
which sends each flag to the vector it contains. Clearly it is a fibration. However, it does more: it respects the positive structures of these spaces, and in some sense their cluster structures as well. We will encapsulate the good properties of the map $\pi$ by saying that it is a \emph{positive fibration}. For the purposes of this paper, let us define this notion as follows.

\begin{definition} Let $Y_1$ and $Y_2$ be two positive varieties. We will say that the map $\pi: Y_1 \rightarrow Y_2$ is a positive fibration if it satisfies the following properties:
\begin{enumerate}
\item $\pi$ respects positive structures: if $x \in Y_1$ is a positive point, then $\pi(x) \in Y_2$ is a positive point. Alternatively, if $f \in \cO(Y_2)$ is a positive rational function, then $\pi^*(f) \in \cO(Y_1)$ is a positive rational function.
\item There exists a positive chart for $Y_2$ consisting of functions $\{f_i\}$ such that the set of functions $\{\pi^*(f_i)\}$ can be extended to a positive chart for $Y_1$.
\end{enumerate}
\end{definition}

We use this definition tentatively, because we are not sure sure if it is the ``right'' notion. It is possible that the last condition is too strong, or could be replaced by an equivalent condition which is easier to verify. We hope that in the future others will become interested in this notion and possibly refine its definition.

In the examples that will be of interest to us, positive fibrations come from maps which are \emph{cluster fibrations}.

\begin{definition} Let $Y_1$ and $Y_2$ be varieties which admit cluster structures. We will say that the map $\pi: Y_1 \rightarrow Y_2$ is a cluster fibration if there exists a cluster for $Y_1$ consisting of functions $\{f_i\}$ and a cluster for $Y_2$ consisting of functions $\{g_i\}$ such that $\{\pi^*(g_i)\} \subset \{f_i\}$. In other words, the pull-back of a cluster on $Y_2$ is contained within a cluster for $Y_1$.
\end{definition}

Here, too, it is possible that one could give a more refined definition. For example, we could further require compatibility between the quivers for these clusters, i.e., the $B$-matrix for the cluster on $Y_2$ is the corresponding submatrix of the $B$-matrix for the cluster on $Y_1$. This will be the case for all the examples that we consider. Many examples of cluster fibrations occur from the process of amalgamation, which is used to describe cluster structures on various Bruhat cells as well as on moduli of decorated local systems \cite{BFZ}, \cite{FG1}, \cite{Le2}.

It is straightforward to verify the following:

\begin{prop} Let $Y_1$ and $Y_2$ be varieties which admit cluster structures. These cluster structures determine positive structures on $Y_1$ and $Y_2$. If $\pi: Y_1 \rightarrow Y_2$ is a cluster fibration, then it is a positive fibration.
\end{prop}

Let us first give some examples of maps that are positive fibrations in our sense.

\begin{prop}\label{posfib} The following are cluster fibrations, and hence positive fibrations:
\begin{enumerate}
\item The natural maps 
$$\pi_i: \Conf_n V \rightarrow \Conf_{n-1} V$$
for $1 \leq i \leq n$ coming from forgeting one of the $n$ vectors.
\item The natural maps 
$$\pi_i: \Conf_n \A \rightarrow \Conf_{n-1} \A$$
for $1 \leq i \leq n$ coming from forgeting one of the $n$ flags.
item
\item The map
$$\pi: \Conf_n \A \rightarrow \Conf_n V$$
sending each flag to the first vector in the flag.
\end{enumerate}
\end{prop}

\begin{proof} The first two claims are easy.

The first claim follows from direct inspection of the coordinate charts that we introduced on $\Conf_n V$. For example, Figure~1 contained the cluster structure for $\Gr(4,8)$. Below we depict the cluster structure for $\Gr(4,9)$. It contains Figure 1 after removing the first column. 

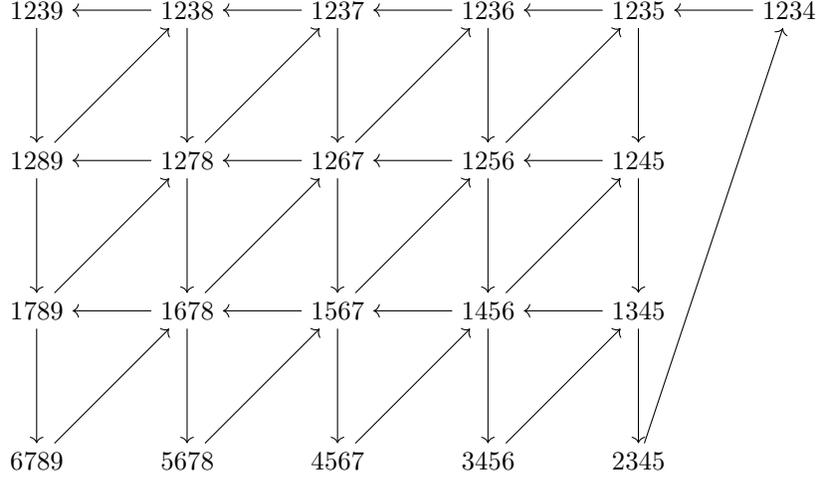
\begin{figure}
\begin{tikzpicture}[scale=2]
\begin{scope}[xshift=-1cm]
  \node[] (x-10) at (-1,0) {$1239$};
  \node[] (x00) at (0,0) {$1238$};
  \node[] (x10) at (1,0) {$1237$};
  \node[] (x20) at (2,0) {$1236$};
  \node[] (x30) at (3,0) {$1235$};
  \node[] (x40) at (4,0) {$1234$};

  \node[] (x-11) at (-1,-1) {$1289$};
  \node[] (x01) at (0,-1) {$1278$};
  \node[] (x11) at (1,-1) {$1267$};
  \node[] (x21) at (2,-1) {$1256$};
  \node[] (x31) at (3,-1) {$1245$};

  \node[] (x-12) at (-1,-2) {$1789$};
  \node[] (x02) at (0,-2) {$1678$};
  \node[] (x12) at (1,-2) {$1567$};
  \node[] (x22) at (2,-2) {$1456$};
  \node[] (x32) at (3,-2) {$1345$};

  \node[] (x-13) at (-1,-3) {$6789$};
  \node[] (x03) at (0,-3) {$5678$};
  \node[] (x13) at (1,-3) {$4567$};
  \node[] (x23) at (2,-3) {$3456$};
  \node[] (x33) at (3,-3) {$2345$};

  \draw [->] (x40) to (x30);
  \draw [->] (x30) to (x20);
  \draw [->] (x20) to (x10);
  \draw [->] (x10) to (x00);
  \draw [->] (x00) to (x-10);
  \draw [->] (x31) to (x21);
  \draw [->] (x21) to (x11);
  \draw [->] (x11) to (x01);
  \draw [->] (x01) to (x-11);
  \draw [->] (x32) to (x22);
  \draw [->] (x22) to (x12);
  \draw [->] (x12) to (x02);
  \draw [->] (x02) to (x-12);

  \draw [->] (x-12) to (x-13);
  \draw [->] (x-11) to (x-12);
  \draw [->] (x-10) to (x-11);
  \draw [->] (x02) to (x03);
  \draw [->] (x01) to (x02);
  \draw [->] (x00) to (x01);
  \draw [->] (x12) to (x13);
  \draw [->] (x11) to (x12);
  \draw [->] (x10) to (x11);
  \draw [->] (x22) to (x23);
  \draw [->] (x21) to (x22);
  \draw [->] (x20) to (x21);
  \draw [->] (x32) to (x33);
  \draw [->] (x31) to (x32);
  \draw [->] (x30) to (x31);

  \draw [->] (x-13) to (x02);
  \draw [->] (x-12) to (x01);
  \draw [->] (x-11) to (x00);
  \draw [->] (x03) to (x12);
  \draw [->] (x02) to (x11);
  \draw [->] (x01) to (x10);
  \draw [->] (x13) to (x22);
  \draw [->] (x12) to (x21);
  \draw [->] (x11) to (x20);
  \draw [->] (x23) to (x32);
  \draw [->] (x22) to (x31);
  \draw [->] (x21) to (x30);
  \draw [->] (x33) to (x40);

\end{scope}


\end{tikzpicture}
\caption{\label{Gr49seed} The cluster structure for $\tGr(4,9)$.}
\end{figure}

Thus it is clear that the map $\pi_9: \Gr(4,9) \rightarrow \Gr(4,8)$ is a cluster fibration. The generalization to general $k$ and $n$ should be clear: the map $\pi_n: \Conf_n V \rightarrow \Conf_{n-1} V$ is a cluster fibration. The same follows for the other maps $\pi_i$ after using the twisted cyclic shift.

The second claim follows similarly. Recall that the cluster structure on $\Conf_n \A$ comes from a triangulation of an $n$-gon. Consider any triangulation of the $n$-gon that uses the triangle with vertices labeled $1, n-1, n$. Then it will contain a triangulation of the $n-1$-gon with vertices $1, 2, \dots, n-1$. Thus it is clear that the map $\pi_n: \Conf_n \A \rightarrow \Conf_{n-1} \A$ is a cluster fibration. Again, we can use the twisted cyclic shift to get the same result for other $\pi_i$.

The third claim is somewhat more difficult, in that it relies on one of the important results of \cite{OPS}. First of all, we need that the space $\Conf_n \A$ is birational to a positroid stratum $\Pi \subset \tGr(k,nk)$. (This was explained, for example, in \cite{MuS}, but was known by others before that paper). The birational map comes from the fact that these spaces have identical cluster structures. We recall that the frozen variables for the cluster structure on $\Pi$ are given by the Grassmann necklace for $\mathcal{M}$, which we denote by $\vec{\mathcal{I}}$. We need the following theorem:

\begin{theorem} (\cite{OPS})  Any weakly separated collection $\mathcal{C}$ of Pl\"ucker coordinates lying in a positroid $\mathcal{M}$, such that $\mathcal{C} \cup \mathcal{I}$ is also weakly separated, can be completed to a cluster on the associated positroid stratum $\Pi$.
\end{theorem}

Now take any cluster on $\Conf_n V$ consisting of functions $\{f_i\}$ which are all Pl\"ucker coordinates (there are many such clusters). The Pl\"ucker coordinates $\{f_i\}$ are a weakly separated collection, because they are a cluster in $\tGr(k,n)$. Next observe that the functions $\{\pi^*(f_i)\}$ are all non-zero on $\Conf_n \A$, and hence belong to the positroid $\mathcal{M}$ associated to the positroid stratum $\Pi$. The map $\pi: \Pi \rightarrow \Conf_n V$ comes from forgetting certain vectors. 
Thus, because the collection $\{f_i\}$ is weakly separated, so is the collection $\{\pi^*(f_i)\}$. It is also straightforward to verify that this collection is weakly separated from the frozen variables. Explicitly, every frozen variable for $\Pi$ is a disjoint union of two cyclic intervals $[i,\dots,i+a-1] \cup [j,\dots,j+b-1]$ where $a+b=k$ and $j = i+k$, 
and any of the Pl\"ucker coordinates $\pi^*(f_i)$ can only intersect the cyclic interval $[i,j]$ at its endpoints. It follows that $\pi^*(f_i)$ is weakly separated from $\vec{\mathcal{I}}$, and thus can be extended to a cluster on $\Pi$.  
\end{proof}

\subsection{Positive fibrations and tropical geometry}

Let us deduce some consequences of the fact that a map $\pi: Y_1 \rightarrow Y_2$ is a positive fibration.

\begin{lemma}\label{welldefined} There is a well-defined tropicalization of the map $\pi$:
$$\pi^t: Y_1(\Z^t) \rightarrow Y_2(\Z^t).$$
\end{lemma}

\begin{proof} This follows directly from the first part of the definition of a positive fibration. Any coordinate function in any positive chart on $Y_2$ pulls back to a positive rational function in any coordinate chart on $Y_1$. Replace $+, \times, \div$ by $\max, +, -$ to get the corresponding tropical expressions.
\end{proof}

We will need a refinement of the above lemma. Recall that for any positive variety $Y$, we can relate the sets $Y(\cK_{>0})$ and $Y(\Zt)$ by the map $-\val$:
$$-\val:  Y(\cK_{>0}) \rightarrow Y(\Zt).$$

\begin{prop}\label{wellbehaved} Suppose that we have $y \in Y_2(\cK_{>0})$ such that $-\val(y) = y^t$. Moreover, suppose that under the map
$$\pi^t: Y_1(\Z^t) \rightarrow Y_2(\Z^t),$$
we have that $\pi^t(x^t)=y^t$. Then there exists $x \in Y_1(\cK_{>0})$ such that $\pi(x)=y$.
\end{prop}

\begin{proof} For this proposition, we will use the full strength of our definition for a positive fibration. The proof is straightforward.

Choose coordinate charts for $Y_1$ and $Y_2$ consisting of functions $\{f_i\}$ and $\{g_j\}$, respectively, such that $\{\pi^*(g_j)\} \subset \{f_i\}$. Because we want $\pi(x)=y$, we need to have that
$$g_j(y)=\pi^*(g_j)(x).$$
This determines the coordinates of the point $x$ at the coordinates $\pi^*(g_j)$. Note that this is consistent with the fact that
$$-\val(g_j(y))=g_j^t(y^t),$$
and
$$-\val(\pi^*(g_j)(x))=\pi^*(g_j)^t(x^t)=g_j^t(y^t).$$
Now the other coordinates in the set $\{f_i\} \backslash \{\pi^*(g_j)\}$ can be set to any values with the appropriate valuations. More precisely, let $f \in \{f_i\} \backslash \{\pi^*(g_j)\}$. Then we can set $f(x)$ to be any element of $\cK_{>0}$ such that
$-\val (f(x))=f^t(x^t)$.

\end{proof}

Note that the above proof also proves the following, which is not immediately obvious:

\begin{cor} Suppose that $\pi: Y_1 \rightarrow Y_2$ is a positive fibration. Then the map $\pi^t: Y_1(\Z^t) \rightarrow Y_2(\Z^t)$ is surjective.
\end{cor}

Let us now specialize to the map $\pi: \Conf_n \A \rightarrow \Conf_n V$.

\begin{cor} There is a well-defined, surjective tropicalization of the map $\pi$:
$$\pi^t: \Conf_n \A(\Z^t) \rightarrow \Conf_n V(\Z^t).$$

Moreover, suppose that we have $y \in \Conf_n V(\cK_{>0})$ such that $-\val(y) = y^t$. Moreover, suppose that under the map
$$\pi^t: \Conf_n \A(\Z^t) \rightarrow \Conf_n V(\Z^t),$$
we have that $\pi^t(x^t)=y^t$. Then there exists $x \in \Conf_n \A(\cK_{>0})$ such that $\pi(x)=y$.
\end{cor}

\subsection{Alternative approach}

In this section, explain how to deduce the final corollary of the last section without using that $\pi: \Conf_n \A \rightarrow \Conf_n V$ is a cluster fibration or a positive fibration. We do this because the fact that $\pi$ is a cluster fibration follows from the deep results of \cite{OPS}. Moreover, we would like to give an argument which is concrete and geometric and which may be applicable in other contexts.

Note that a pair of generic flags for $SL_k$ determines $k$ vectors. This means that $\Conf_{nk} V$ and $\Conf_{2n} \A$ are very closely related. In fact, they are ``almost'' isomorphic cluster algebras-- they have the same $B$-matrices, and there is a natural bijective correspondence between their respective clusters and cluster variables-- differing only in details involving frozen variables \cite{F1}. For example,  $\Conf_{2n} \A$ has more frozen variables than $\Conf_{nk} V$. More precisely, consider the subvariety $\overline{\Conf}_{nk} V \subset \Conf_{nk} V$ consisting of configurations of $nk$ vectors labeled $v_{11}, \dots, v_{1k},$ $v_{21}, \dots, v_{2k},$ $\dots, v_{n1}, \dots, v_{nk}$, such that $\det(v_{i1}, \dots, v_{ik})=1$. It inherits a positive structure from $\Conf_{nk} V$. Then there are maps in both directions $\overline{\Conf}_{nk} V \leftrightarrow \Conf_{2n} \A$, sending positive points in one space to positive points in the other. In other words, we have an isomorphism of positive varieties between $\overline{\Conf}_{nk} V$ and $\Conf_{2n} \A$.

Let us explain the maps in both directions. Suppose that we have a configuration of $nk$ vectors that lies in $\overline{\Conf}_{nk} V$. We will label them $v_{11}, \dots, v_{1k},$ $v_{21}, \dots, v_{2k},$ $\dots, v_{n1}, \dots, v_{nk}$ as before. From these vectors, we can form the flags
$$F_1=(v_{11}, \dots, v_{1k})$$
$$F_2=(v_{1k}, \dots, v_{11})$$
$$F_3=(v_{21}, \dots, v_{2k})$$
$$F_4=(v_{2k}, \dots, v_{21})$$
$$\dots$$
$$F_{2n-1}=(v_{n1}, \dots, v_{nk})$$
$$F_{2n}=(v_{nk}, \dots, v_{n1}).$$


In the reverse direction, given $F_1, \dots, F_{2n}$, the intersection of $F_{2i-1}$ and $F_{2i}$ generically determines a line, and there is a unique $\SL(V)$-equivariant way to choose a point in this line.  
Computing these intersections for $i=1,\dots,k$ provides the vectors $v_{i1}, \dots, v_{ik}$. Because we are dealing with $\SL_k$ flags, we automatically have $\det(v_{i1}, \dots, v_{ik})=1$. The composite of these two maps is not the identity, but it rescales every cluster variable by some monomial in the frozen variables. 

Nonetheless, these maps preserve the positive structures on $\Conf_{2n} \A$ and $\overline{\Conf}_{nk} V$.  This means that a positive configuration of $nk$ vectors $v_{11}, \dots, v_{1k},$ $v_{21}, \dots, v_{2k},$ $\dots, v_{n1}, \dots, v_{nk}$, i.e., a point of $\overline{\Conf}_{nk} V(\R_{>0})$, determines a positive configuration of $2n$ flags $F_1, \dots, F_{2n}$, i.e., a point of $\Conf_{nk} \A(\R_{>0})$, and vice versa.

We can then use this to show that
$$\pi^t: \Conf_n \A(\Z^t) \rightarrow \Conf_n V(\Z^t)$$
is well-defined. Recall that in order to do this, we need to show that if $f$ is a positive function on $\Conf_n V$, it pulls back to a positive rational function on $\Conf_n \A$.

We will need one other fact.
\begin{lemma}
If $\pi: Y_1 \rightarrow Y_2$ is a positive fibration, then if $f \in \cO(Y_2)$ and $\pi^*(f)$ is a positive rational function on $Y_1$, then $f$ is a positive rational function on $Y_2$.
\end{lemma}

\begin{proof} Take a positive chart $\{g_j\}$ for $Y_1$ such that $\{\pi^*(g_j)\}$ extends to a positive chart on $Y_2$. In this chart, $\pi^*(f)$ is a positive rational function. Because it descends to a function on $Y_1$ it only depends on the functions $\{\pi^*(g_j)\}$. Because  $\pi^*(f)$ is a positive rational function of the coordinates $\{\pi^*(g_j)\}$, $f$ is a positive rational function of the coordinates $\{g_j\}$.
\end{proof}

Now let $f$ be a positive rational function on $\Conf_n V$. We want to show that $\pi^*(f)$ is a positive rational function on $\Conf_n \A$. There is a forgetful map
$$\phi: \overline{\Conf}_{nk} V \rightarrow \Conf_n V$$
sending $v_{11}, \dots, v_{1k},$ $v_{21}, \dots, v_{2k},$ $\dots, v_{n1}, \dots, v_{nk}$ to $v_{11}, v_{21}, \dots, v_{n1}.$ This map is a positive fibration from Proposition~\ref{posfib}, so $\phi^*(f)$ is a positive rational function on $\overline{\Conf}_{nk} V$.

The map $\phi$ is a composite of maps
$$\overline{\Conf}_{nk} V \to \Conf_{2n} \A \xrightarrow{\phi'} \Conf_n \A \xrightarrow{\pi} \Conf_n V.$$
Here, the map $\phi'$ is a forgetful map $F_1, \dots, F_{2n} \mapsto F_1, F_3, \dots, F_{2n-1}$ and is also a positive fibration by Proposition~\ref{posfib}. Thus $\phi^*(f)$ can also be viewed as a positive rational function on $\Conf_{2n} \A$. However, it is also pulled back from a function on $\Conf_n \A$, which must be positive by the lemma, so we are done. 

Finally, suppose that we have $y \in \Conf_n V(\cK_{>0})$ such that $-\val(y) = y^t$. Moreover, suppose that under the map
$$\pi^t: \Conf_n \A(\Z^t) \rightarrow \Conf_n V(\Z^t),$$
we have that $\pi^t(x^t)=y^t$. Let us now show that there exists $x \in \Conf_n \A(\cK_{>0})$ such that $\pi(x)=y$.

Because $\phi'$ is a positive fibration, there exists a point $\widetilde{x^t} \in \Conf_{2n} \A(\Z^t)$ such that $\phi'^t(\widetilde{x^t})=x^t$. Now $\widetilde{x^t}$ is also a point of 
$\overline{\Conf}_{nk} V(\Zt)$. We can then apply Proposition~\ref{wellbehaved} to conclude that there exists $\widetilde{x} \in \Conf_{nk} V(\cK_{>0})$ such that $\phi(\widetilde{x})=y$. Then $x := \phi'(\widetilde{x}) \in \Conf_n \A(\cK_{>0})$ is the point we seek.

Thus, using only the first two easy parts of Proposition~\ref{posfib} we are still able to prove our main results from the previous section, namely that $\pi^t: \Conf_n \A(\Z^t) \rightarrow \Conf_n V(\Z^t)$ is well defined, and that $\cK_{>0}$-valued points of $\Conf_n V$ can be lifted to $\cK_{>0}$-valued points of $\Conf_n \A$ with any specified valuations given the constraints of the map $\pi^t$.

\medskip

We will later apply what we know about the map
$\pi^t: \Conf_n \A(\Z^t) \twoheadrightarrow \Conf_n V(\Z^t)$
to describe the tropicalization of $\tGr(k,n) \simeq \Conf_n V$. It was shown in \cite{Le} that the tropical points of $\Conf_n \A$ parameterize positive configurations of points in the affine building, objects we called higher laminations. It then follows that $\Conf_n V(\Z^t)$ parameterizes some equivalence classes of higher laminations.

However, there is a dual perspective. Instead of viewing $\Conf_n V(\Z^t)$ as a quotient of  $\Conf_n \A(\Z^t)$, we can view it as a subset. In other words, there is a distinguished section
$$\Conf_n V(\Z^t) \hookrightarrow \Conf_n \A(\frac{1}{k}\Z^t).$$
This does not come from any map from $\Conf_n V$ to $\Conf_n \A$. It comes from the duality conjectures, which tell us that the appropriate tropical points of $\Conf_n V$ and $\Conf_n \A$ parameterize functions on $\Conf_n V$ and $\Conf_n \A$, respectively. The natural inclusion
$$\pi^*: \cO(\Conf_n V) \hookrightarrow \cO(\Conf_n \A).$$
then gives a natural map between the tropical spaces.

We expect that our approach to understanding tropical points of $\Conf_n V$ can be extended. Suppose that we have parabolic subgroups $P_1, P_2, \dots, P_n$. Suppose that $P_i = M_iA_iN_i$ where $M_i, A_i, N_i$ are the semisimple, abelian and nilpotent subgroups in the Langlands decomposition of $P_i$. Then the space

\begin{equation*}
 \raisebox{-.2em}{$G$} \Big\backslash \prod_{i=1}^n \Big( \raisebox{.2em}{$G$}\left / \raisebox{-.2em}{$(M_iN_i)$}  \Big) \right.
\end{equation*}
should have a positive structure such that the natural map
\begin{equation*}
\Conf_n \A \longrightarrow  \raisebox{-.2em}{$G$} \Big\backslash \prod_{i=1}^n \Big( \raisebox{.2em}{$G$}\left / \raisebox{-.2em}{$(M_iN_i)$}  \Big) \right.
\end{equation*}
is a positive fibration. Note that the above map comes from the fact that $\A = G/U$ has a natural map to $G/(M_iN_i)$.

\section{Tropicalization of \texorpdfstring{$\Conf_n \A$}{the configuration space of flags}}

In \cite{Le}, it was shown that the space $\Conf_n \A$ parameterizes higher laminations. In this section we give a self-contained exposition of higher laminations. Higher laminations are defined as certain configurations of points in the affine building. Thus we will start by defining the the affine Grassmannian and the affine building.

\subsection{Affine Grassmannian and affine buildings}

Let $G$ be a simple, simply-connected complex algebraic group and let $\vG$ be its Langlands dual group. Let $\F$ be a field, which for our purposes will always be $\R$ or $\C$. Let $\cO = \F[[t]]$ be the ring of formal power series over $\F$. It is a valuation ring, where the \emph{valuation} $\val(x)$ of an element $$x = \sum_i a_i t^i \in \F[[t]]$$ is the minimum $i$ such that $a_i \neq 0$. Let $\cK = \F((t))$ be the fraction field of $\cO$. Then 
$$\Gr(\F) = \Gr(G) = G(\cK)/G(\cO)$$ is the set of $\F$-points of the \emph{affine Grassmannian} for $G$.

In this paper, we will always denote the affine Grassmannian by $\Gr$, while the usual, finite-dimensional Grassmannian will always be denoted with extra parameters $k$ and $n$, as in $\Gr(k,n)$.

Let us compare the affine Grassmanian for $\GL_k$, $\PGL_k$ and $\SL_k$. For $G=\GL_k$, a point in the affine Grassmannian corresponds to a finitely generated, rank $k$, $\cO$-submodule of $\cK^k$, i.e., a lattice in $\cK^k$. For $G=\SL_k$, a point in the affine Grassmannian corresponds to a finitely generated, rank $k$, $\cO$-submodule of $\cK^k$ such that there exist generators $v_1, \dots, v_k$ for this submodule such that $$v_1 \wedge \dots \wedge v_k=e_1 \wedge \dots \wedge e_k,$$ where $e_1, \dots, e_k$ is the standard basis of $\cK^k$. For $G=\PGL_k$, a point in the affine Grassmannian corresponds to an equivalence class of lattices up to scale: we say lattices $L$ and $L'$ are equivalent if $L=cL'$ for some $k \in \cK$. In all three cases, the affine Grassmannian consists of some space of lattices. Moreover, $G(\cK)$ acts on the space of such lattices, and the stabilizer of each lattice is isomorphic to $G(\cO)$.

The affine Grassmannian $\Gr$ also has a metric valued in dominant coweights: the set of pairs of elements of $\Gr$ up to the action of $G(\cK)$ is exactly the set of double cosets $$G(\cO) \backslash G(\cK) / G(\cO).$$ These double cosets, in turn, are in bijection with the cone $\Lambda_+$ of dominant coweights of $G$. Recall that the coweight lattice $\Lambda$ is defined as $\mathrm{Hom}(\mathbf{G}_m,T)$. The coweight lattice contains dominant coweights, those coweights lying in the dominant cone. For example, for $G=\GL_k$, the set of dominant coweights is exactly the set of $\mu=(\mu_1, \dots, \mu_k)$, where $\mu_1 \geq \mu_2 \geq \cdots \geq \mu_k$ and $\mu_i \in \Z$. For $G=\SL_k$, the dominant coweights are the subset of $\Z^k$ where $\mu_1+\cdots+\mu_k=0$. For $G=\PGL_k$, the dominant coweights are the lattice $\Z^k/(1, 1, \dots, 1)$.

Let us explain why the set of double cosets is in bijection with the set of dominant coweights.

Fix a basis $e_1, \dots, e_k$ of $\cK^k$. Then given any dominant coweight $\mu$ of $G$, there is an associated point $t^\mu$ in the (real) affine Grassmannian: to a coweight $\mu=(\mu_1, \dots, \mu_k)$ we associate the element of $G(\cK)$ with diagonal entries $t^{-\mu_i}$, and then apply that element to the trivial lattice $<e_1, \dots, e_k>$. Any two points $p$ and $q$ of the affine Grassmannian can be translated by an element of $G(\cK)$ to $t^0$ and $t^\mu$, respectively, for some unique dominant coweight $\mu$. This gives the identification of the double coset space with $\Lambda_+$.

Under this circumstance, we will write $$d(p,q) = \mu$$ and say that the distance from $p$ to $q$ is $\mu$.

Let us collect some facts about this distance function $d$. This distance function is not symmetric; one can easily check that $$d(q,p)=-w_0 d(p,q)$$ where $w_0$ is the longest element of the Weyl group of $G$. However, there is a partial order on $\Lambda$ defined by $\lambda > \mu$ if $\lambda - \mu$ is positive (i.e., in the positive span of the positive co-roots). Under this partial ordering, the distance function satisfies a version of the triangle inequality. By construction, the action of $G(\cK)$ on the affine Grassmannian preserves this distance function.

We can now introduce the affine building for $G=\PGL_k$, the case which is of the most interest to us. The affine building is a simplicial complex. The set of vertices of the affine building for $\PGL_k$ is precisely given by the points of the affine Grassmannian $\Gr(\PGL_k)$.

The $k$-simplices of the affine building are given as follows: for any lattices $L_0, L_1, \dots, L_k$, there is a $k$-simplex with vertices at $L_0, L_1, \dots, L_k$ if and only if (replacing each lattice by an equivalent one if necessary) $$L_0 \subset L_1 \subset \cdots \subset L_k \subset t^{-1}L_0.$$
We will sometimes restrict our attention to those vertices of the building that come from points in the affine Grassmannian for $\SL_k$.

The non-symmetric, coweight-valued metric we defined above descends from the affine Grassmannian to the affine building. The notion of a \emph{geodesic} with respect to this metric is sometimes useful. For our purposes, a geodesic in the building is a path that travels along edges in the building from vertex to vertex, such that the sum of the distances from vertex to vertex is minimal (with respect to the partial order defined above). It is a property of affine buildings that geodesics exist. Note that in general there will be many geodesics between any two points.

\subsection{Tropicalization of Functions}

Let us now describe the tropicalization $f_{i_1i_2i_3}^t$ of the functions $f_{i_1i_2i_3}$ introduced in Section \ref{flags}. Each $f_{i_1i_2i_3}^t$ is a function on the set of configurations of three points in the affine Grassmannian for $\SL_k$. We will later explain how configurations of $m$ points in the affine Grassmannian give tropical points of the space of configurations of $m$ principal flags. The functions $f_{i_1i_2i_3}^t$ will then be the tropicalization of the functions $f_{i_1i_2i_3}$. The functions $f_{i_1i_2i_3}^t$ are the same as the functions $H_{i_1i_2i_3}$, which were defined in a slightly different way in \cite{K}.

Let $x_1, x_2, x_3$ be three points in the affine Grassmannian for $\SL_k$, thought of as $\cO$-submodules of $\cK^k$. For $i_1, i_2, i_3$ such that $i_1 + i_2 + i_3=k$, we will consider the quantity

$$-\val(\det(u_1, \dots, u_{i_1}, v_1, \dots v_{i_2}, w_1, \dots, w_{i_3}))$$
as $u_1, \dots, u_{i_1}$ range over elements of the $\cO$-submodule $x_1$, $v_1, \dots v_{i_2}$ range over elements of the $\cO$-submodule $x_2$, and $w_1, \dots, w_{i_3}$ range over elements of the $\cO$-submodule $x_3$. Define $f_{i_1i_2i_3}^t (x_1,x_2,x_3)$ as the maximum value of of this quantity, i.e., the largest value of$$-\val(\det(u_1, \dots, u_{i_1}, v_1, \dots v_{i_2}, w_1, \dots, w_{i_3}))$$ as all the vectors $u_1, \dots, u_{i_1},$ $v_1, \dots v_{i_2},$ $w_1, \dots, w_{i_3}$ range over elements of the respective $\cO$-submodules $x_1, x_2, x_3$.

\begin{rmk} It is not hard to check that the edge functions recover the distance between two points in the affine Grassmannian (and hence also the affine building). More precisely, $f_{ij0}^t (x_1,x_2,x_3)$ is given by $\omega_j \cdot d(x_1,x_2)=\omega_i \cdot d(x_2,x_1)$ where $\omega_i$ is a fundamental weight for $SL_k$.
\end{rmk}

We can naturally extend $f_{i_1i_2i_3}^t$ to a function on the affine Grassmannian for $\PGL_k$. This will become useful later when we construct distinguished representatives for tropical points of~$X$. Let $x_1, x_2, x_3$ be three points in the affine Grassmannian for $\PGL_k$, represented by three lattices $L_1, L_2, L_3$. For $i_1, i_2, i_3$ as above, we can again minimize
$$-\val(\det(u_1, \dots, u_{i_1}, v_1, \dots v_{i_2}, w_1, \dots, w_{i_3}))$$
as $u_1, \dots, u_{i_1}$ range over elements of the $\cO$-submodule $L_1$, $v_1, \dots v_{i_2}$ range over elements of the $\cO$-submodule $L_2$, and $w_1, \dots, w_{i_3}$ range over elements of the $\cO$-submodule $L_3$. Call the resulting function 
$$\tilde{f}_{i_1i_2i_3}^t (L_1,L_2,L_3).$$ Note that $\tilde{f}_{i_1i_2i_3}^t (L_1,L_2,L_3)$ will depend on the representative lattices $L_1, L_2, L_3$ that we chose, which are only determined up to scale. To fix this, we put
\begin{equation} \label{PGLfunction}
f_{i_1i_2i_3}^t (x_1,x_2,x_3)=\tilde{f}_{i_1i_2i_3}^t (L_1,L_2,L_3)+\frac{\val(\det(L_1)\det(L_2)\det(L_3))}{k}.
\end{equation}
This will now be a function of three points in the affine Grassmannian for $PGL_k$. Notice if $L_1, L_2, L_3$ have determinant $1$ (and hence determine a point in the affine Grassmannian for $SL_k$) our definition coincides with the definition above.

It was shown in \cite{Le3} that there is another, purely metric, way to describe the functions $f_{i_1i_2i_3}^t$. First, define
$$d_i(p,x) = \omega_i \cdot d(p,x).$$
Then let $i_1+i_2+\cdots+i_m=k$. For points $x_1, x_2, \dots, x_m$ in the affine building for $\PGL_k$, we can define the functions
$$f_{i_1i_2\dots i_m}^t(x_1,x_2\dots,x_m).$$
Then we have that
\begin{theorem}\label{basicfn}
$$f_{i_1i_2 \dots i_m}^t(x_1,x_2, \dots, x_m)= \min_{p} \left(d_{i_1}(p,x_1) + d_{i_2}(p,x_2) + \cdots + d_{i_m}(p,x_m) \right)$$
where the minimum is taken over all $p$ in the affine Grassmannian for $\PGL_k$. A similar metric interpretation holds in the affine building for $\SL_k$. 
\end{theorem}

\subsection{Configurations of flags and their tropicalization}

We now recall the description of the tropicalization of $\Conf_n \A$. We begin by defining positive configurations of points in the affine Grassmannian. The treatment here is adapted from \cite{GS} and \cite{Le}.

\begin{definition} \label{posconfig} Let $x_1, x_2, \dots x_n$ be $n$ points of the real affine Grassmannian. Then $x_1, x_2, \dots x_n$ will be called a positive configuration of points in the affine Grassmannian if and only if there exist ordered bases for $x_i$, 
$$v_{i1}, v_{i2}, \dots, v_{ik},$$ such that for each $1 \leq p < q <r \leq n$, and each triple of non-negative integers $i_1, i_2, i_3$ such that $i_1+i_2+i_3=k$,
\begin{enumerate}
\item $f_{i_1i_2i_3}^t (x_p,x_q,x_r) = -\val(\det(v_{p1}, \dots, v_{pi_1}, v_{q1}, \dots v_{qi_2}, v_{r1}, \dots, v_{ri_3}))$
\item the leading coefficient of $\det(v_{p1}, \dots, v_{pi_1}, v_{q1}, \dots v_{qi_2}, v_{r1}, \dots, v_{ri_3})$ is positive.
\end{enumerate}
\end{definition}

We will call the first condition the \emph{valuation minimizing} property of the bases $v_{i1}, v_{i2}, \dots, v_{ik}$ with respect to the configuration $x_1, x_2, \dots x_n$.

The above definition works for both $\SL_k$ and $\PGL_k$. Note that it is important in the above definition that we are taking the valuations of the determinants of the first $i_1$ (respectively $i_2, i_3$) vectors among the bases for $x_p$ (respectively $x_q, x_r$), and not just any $i_1$ (respectively $i_2, i_3$) vectors.

\begin{rmk} It is sufficient to verify the two conditions above for only those triples $p, q, r$ occuring in any given triangulation of the $m$-gon. The valuation condition and the positivity condition for one triangulation implies the these conditions for any other triangulation, and hence for an arbitrary triple $p, q, r$.
\end{rmk}

Suppose that $x_1, x_2, \dots x_n$ are a positive configuration of points in the affine Grassmannian. Points in the affine Grassmannian give vertices in the affine building, so that $x_1, x_2, \dots x_n$ can also be viewed as a \emph{positive configuration of points in the affine building}. This is more natural for us because the invariants of the configuration $x_1, x_2, \dots x_n$ (like $f_{i_1i_2i_3}^t$) will be metric in nature. Moreover, the piecewise-linear structure of the building reflects more closely the piecewise-linear structure in tropical geometry.

Higher laminations are roughly given by positive configurations of points in the affine building. More precisely, higher laminations are in bijection with \emph{virtual} positive configurations of points in the affine building. Let us define this notion.

A virtual positive configurations of $n$ points in the affine building is given by the data of $n$ ordered pairs
$$(x_1,\lambda_1), (x_2,\lambda_2), \dots (x_n,\lambda_n)$$
where $x_1, x_2, \dots x_n$ is a positive configuration of points in the affine building, and $\lambda_1, \dots, \lambda_n$ are coweights for $\SL_k$ (respectively, $\PGL_k$).

We can extend the functions $f_{i_1i_2i_3}^t$ to virtual configurations by the following formula:
$$f_{i_1i_2i_3}^t ((x_p,\lambda_p),(x_q,\lambda_q),(x_r,\lambda_r))=f_{i_1i_2i_3}^t (x_p,x_q,x_r) + \omega_{i_1}\cdot\lambda_p+\omega_{i_2}\cdot\lambda_q+\omega_{i_3}\cdot\lambda_r.$$
Suppose that the coweights $\lambda_1, \dots, \lambda_n$ are dominant. Then, using Theorem~\ref{basicfn}, we can see that the effect of $\lambda_1, \dots, \lambda_n$ is to to move each point $x_p$ a distance $\lambda_p$ farther away from the others:
$$f_{i_1i_2i_3}^t ((x_p,\lambda_p),(x_q,\lambda_q),(x_r,\lambda_r))=\min_{y} \omega_{i_1}\cdot(d(y,x_p)+\lambda_p)+\omega_{i_2}\cdot(d(y,x_q)+\lambda_q) + \omega_{i_3}\cdot (d(y,x_r)+\lambda_r).$$
Of course, when $\lambda_1, \dots, \lambda_n$ are non-dominant, we cannot use such an interpretation.

We will call virtual positive configurations of points 
$$(x_1,\lambda_1), \dots (x_n,\lambda_n)$$ and 
$$(x'_1,\lambda'_1), \dots (x'_n,\lambda'_n)$$ equivalent if for each $1 \leq p < q <r \leq m$, we have
$$f_{i_1i_2i_3}^t((x_p,\lambda_p),(x_q,\lambda_q),(x_r,\lambda_r))=f_{i_1i_2i_3}^t((x'_p,\lambda'_p),(x'_q,\lambda'_q),(x'_r,\lambda'_r)).$$
Note that this can be interpreted as a purely metric condition on the virtual positive configurations. We conjectured in \cite{Le} that non-virtual positive configurations of points are equivalent only if they are isometric (the ``if'' direction is straightforward).

\begin{definition} A higher lamination (on a disc with marked $n$ marked points) is an equivalence class of virtual positive configurations of $n$ points in the affine building up to equivalence.
\end{definition}

\begin{theorem}[\cite{GS}, \cite{Le}] There is a natural bijection between tropical points $\Conf_n \A(\Zt)$ and higher laminations on a disc with $n$ marked points. This bijection is characterized by the fact that the functions $f_{i_1i_2i_3}$ on $\Conf_n \A$ tropicalize to give the functions $f_{i_1i_2i_3}^t$ on higher laminations.
\end{theorem}

\begin{rmk} Let us say something about virtual configurations. The space $\Conf_n \A$ has a natural action of $H^n$. Suppose we have a point of $\Conf_n \A$ given by flags $F_1, \dots, F_n$, where the flag $F_i$ is given by vectors $$v_{i1}, v_{i2}, \dots, v_{ik}.$$ Then there is an action of $H$ on each flag by rescaling the vectors $v_{i1}, v_{i2}, \dots, v_{ik}.$ We denote this action by $h \cdot F$ where $h \in H$ and $F$ is a principal flag. Similarly, if $\overrightarrow{h} \in H^n$, we write $\overrightarrow{h} \cdot x$ for the action of $\overrightarrow{h}$ on $x \in \Conf_n \A$.

The tropicalization of this is an action of $H^n(\Zt)$ on $\Conf_n \A(\Zt)$, where $H^n(\Zt) \cong \bigoplus_{i=1}^n (\mathbb{Z}^k)$. Suppose that a point of $\Conf_n \A(\Zt)$ is given by the virtual configuration $$(x_1,\lambda_1), \dots (x_n,\lambda_n).$$ Then we have an action of $H(\Zt)$ on each $(x_i,\lambda_i)$: we have that $\mu \in H(\Zt)$ takes
$(x_i,\lambda_i)$ to $(x_i,\lambda_i+\mu)$.

This action can be made more transparent in the case that we have an actual configuration $x_1, \dots, x_n$ and dominant coweights $(\mu_1, \dots, \mu_n) \in H^n(\Zt)$. Suppose that in the notation of Definition~\ref{posconfig} we have that $x_i$ has generators $v_{i1}, v_{i2}, \dots, v_{ik}.$ Moreover, suppose the coweight $\mu_i$ as a vector is given by the $k$-tuple $(\mu_{i1}, \dots, \mu_{ik})$. Then the action of $(\mu_1, \dots, \mu_n)$ on the configuration $x_1, \dots, x_n$ takes $x_i$ to the lattice generated by
$$t^{-\mu_{i1}}v_{i1}, t^{-\mu_{i2}}v_{i2}, \dots, t^{-\mu_{ik}}v_{ik}.$$
Note that if the coweights $\mu_i$ are non-dominant, if we try to define the action as above, we may destroy the valuation minimizing property of $v_{i1}, v_{i2}, \dots, v_{ik}.$

For any $x \in \Conf_n \A(\Zt)$, its orbit under the action of $H^n(\Zt)$ is called its \emph{lineality space}. This is the analogue for configurations of flags of the lineality spaces of the tropical Grassmannian. We will later see in Section 6 that the lineality space is related to rescaling by monomials in the frozen variables under the duality conjectures.
\end{rmk}

\subsection{Explicit Construction}\label{construction}

We end this section by discussing how to explicitly construct higher laminations as well as how to evaluate more general tropical functions on them. 

Suppose we want to construct the higher lamination corresponding to a tropical point $x^t \in \Conf_n \A(\Zt)$. There exists a point $x \in \Conf_n \A(\cK_{>0})$ such that $-\val(x)=x^t$. The point $x$ corresponds to a configuration of flags $F_1, \dots, F_n$.

There is a function $\W \in \cO(\Conf_n \A)$ called the \emph{potential}. The tropicalization $\W^t$ of this function has important properties. It gives us a distinguished subset 
\begin{equation}\label{eq:hiveinequalities}
\Conf^+_n \A(\Zt)\{x^t \in \Conf_n \A(\Zt) | \W^t(x^t) \geq 0 \}.
\end{equation} 
This is precisely the set of points such that $x^t$ corresponds to a higher lamination that is a positive configuration of points in the affine building, as opposed to merely being a \emph{virtual} positive configuration of points. Let us say that the point $x^t$ satisfies the \emph{hive inequalities} if $\W^t(x^t) \geq 0$. This terminology comes from a connection to the hives of Knutson and Tao. We review the hive inequalities in Section 6.2. 

Suppose for the moment that our tropical point $x^t$ satisfies the hive inequalities. Then there is a simple procedure for going from the configuration of flags $F_1, \dots, F_n$ to a configuration of points in the affine Grassmannian or building. Take the principal flag $F_1 \in \A$. Take any other flag $F_i$ where $i \neq 1$. Then let $\overline{F_i} \in \B := G/B$ be the corresponding ordinary flag. Because $x \in \Conf_n \A(\cK_{>0})$, the flags $F_1$ and $\overline{F_i}$ are transverse, and therefore determine a frame $v_{11}, \dots v_{1k}$ for the flag $F_1$ (cf.~the discussion in Section 3.2). The $\cO$-module generated by the frame gives a point $x_1 \in \Gr$. It turns out that $x_1$ is independent of of which flag $F_i$ we chose--this is a consequence of the fact that $x^t$ satisfies the hive inequalities.

We can then similarly construct $x_1, \dots, x_n$ to get a configuration of points in the affine Grassmannian. It is not hard to check that the vectors $v_{i1}, \dots v_{ik}$ generating $x_i$ satisfy the valuation minimizing and positivity properties of Definition \ref{posconfig}, \cite{Le}, \cite{GS}.

Now suppose that $x^t$ does not satisfy the hive inequalities. Recall that the group $H^n(\Zt)$ acts on $\Conf_n \A(\Zt)$ by moving points with in their lineality space. There exists $\overrightarrow{\mu}=(\mu_1, \dots, \mu_n) \in H^n(\Zt)$ such that $\overrightarrow{\mu} \cdot x^t$ does satisfy the hive inequalities.

For $\overrightarrow{\mu} \in H^n(\Zt)$ we have a corresponding element 
$$t^{-\overrightarrow{\mu}}=(t^{-\mu_1}, \dots, t^{-\mu_n}) \in H(\cK_{>0}).$$ 
Then $t^{-\overrightarrow{\mu}} \cdot x$ is a point in $\Conf_n \A(\cK_{>0})$ such that 
$$-\val(t^{-\overrightarrow{\mu}} \cdot x) = \overrightarrow{\mu} \cdot x^t.$$
Because $\overrightarrow{\mu} \cdot x^t$ satisfies the hive inequalities, we can use the procedure above to get a positive configuration of points in the affine Grassmannian $y_1, \dots, y_n$. Then the configuration
$$(y_1, -\mu_1), \dots, (y_n, -\mu_n)$$
gives the higher lamination corresponding to the point $x^t$.

\section{Equivalence classes of higher laminations}

Recall from Section \ref{natural} that we have a map $\pi: \Conf_n \A \rightarrow \Conf_n V$ which induces a tropical map
$$\pi^t: \Conf_n \A(\Z^t) \twoheadrightarrow \Conf_n V(\Z^t).$$

For a tropical point $y^t \in \Conf_n V(\Z^t)$, we would like to consider the set $$(\pi^t)^{-1}(y^t) \subset \Conf_n \A(\Z^t).$$
This is an equivalence class of higher laminations. 

Let $x^t$ be a higher lamination such that $\pi^t(x^t)=y^t$. Suppose $x^t$ is given by the virtual configuration $(x_1,\lambda_1), \dots, (x_n,\lambda_n)$, which we will abbreviate $\{(x_i,\lambda_i)\}$. In this situation, we will say that the higher lamination $\{(x_i,\lambda_i)\}$ \emph{represents} or \emph{is a representative for} the tropical point $y^t$.

We would like to have a characterization of the different higher laminations representing $y^t$. Note that by definition, we have that
$$P_I^t(y^t) = f_{1\dots1}^t((x_i)_{i \in I}),$$
which in turn is the minimum value of
$$-\val(\det((u_i)_{i \in I})),$$
where $u_i$ range over vectors in the lattice associated to $x_i$. We also have a metric interpretation of $P_I^t$:

\begin{prop} Suppose that the higher lamination $\{(x_i,\lambda_i)\}$ represents the tropical point $y^t \in \Conf_n V(\Z^t)$. Let $I \subset [n]$. Then
$$P_I^t(y^t) = \min_{p} \sum_{i \in I} (d_1(p,x_i)+\omega_1(\lambda_i)).$$
\end{prop}

This proposition describes the tropicalization of the Pl\"ucker coordinate $P_I$ in terms of the geometry of the building.

\begin{proof} The higher lamination $\{(x_i,\lambda_i)\}$ corresponds to the tropical point $x^t \in \Conf_n \A(\Z^t).$ The pull back of the function $P_I$ to $\Conf_n \A$ is the function $f_{1\dots1}$ associated to the subset $I \subset [n]$. Now apply Theorem \ref{basicfn}.
\end{proof}

Note that $\pi^t(x^t)$ only depends on the values of $\omega_1(\lambda_i)$ and not the coweights $\lambda_i$ themselves. In otherwords, if $\omega_1(\lambda_i)=\omega_1(\lambda'_i)$, then the higher laminations $\{(x_i,\lambda_i)\}$ and $\{(x_i,\lambda'_i)\}$ represent the same point tropical point of $\Conf_n V(\Z^t)$.

Let us point out here that it is not hard to explicitly construct higher laminations representing a particular point $y^t \in \Conf_n V(\Z^t)$. We can find a lift of $y^t$ to $x^t \in \Conf_n \A(\Z^t)$, then construct a higher lamination from $x^t$ as in Section \ref{construction}.

We discuss in the next section how, up to moving $y^t$ in its lineality space, we can represent $y^t$ by a higher lamination which corresponds to an (non-virtual) positive configuration of points in the affine Grassmannian. Therefore, let us assume for now that $y^t$ has a lift $x^t$ corresponding to the positive configurations of points $x_1, x_2, \dots x_n$ in the affine Grassmannian. Then Definition \ref{posconfig} gives us the following:

\begin{prop} There exist generators/vectors $v_i \in x_i$ such that for $I \subset [n]$, we have
\begin{enumerate}
\item $P_I^t (y^t) = -\val(\det(\{v_i\}_{i \in I}))$ and
\item the leading coefficient of $\det((v_i)_{i \in I})$ is positive.
\end{enumerate}
\end{prop}

We will call the first condition the \emph{valuation minimizing} property of the generators $v_{1}, \dots, v_{n}$ with respect to the configuration $x_1, x_2, \dots x_n$. 

\subsection{The lineality space for \texorpdfstring{$\Conf_n V$}{the configuration space of vectors}}

We have already seen the action of $H^n$ on $\Conf_n \A$, and its tropicalization to an action of $H^n(\Zt)$ on $\Conf_n \A(\Zt)$. The orbits of $H^n(\Zt)$ (respectively, $H^n(\Qt)$, $H^n(\Rt)$) are the lineality spaces of $\Conf_n \A(\Zt)$ (respectively, $\Conf_n \A(\Qt)$, $\Conf_n \A(\Rt)$).

This action is compatible with with the map $\pi: \Conf_n \A \rightarrow \Conf_n V$. There is an action of the $n$-dimensional torus $T^n$ on $\Conf_n V$ given by scaling each vector. Moreover, there is a map
$$\omega_1: H \rightarrow T$$
given by the character associated to the weight $\omega_1$. The $n$-fold product of this map is
$$\psi := \omega_1^n: H^n \rightarrow T^n.$$
The map $\psi$ intertwines the action of $H^n$ on $\Conf_n \A$ with the action of $T^n$ on $\Conf_n V$.

The lineality spaces of $\Conf_n V(\Zt)$ (respectively, $\Conf_n V(\Qt)$, $\Conf_n V(\Rt)$) are the orbits of $T^n(\Zt)$ (respectively, $T^n(\Qt)$, $T^n(\Rt)$).

As was the case for $\Conf_n \A$, any tropical point of $y^t \in \Conf_n V(\Z^t)$ has a point in its lineality space that is represented by a higher lamination that is a positive configuration of points in the affine building, as opposed to merely being a \emph{virtual} positive configuration of points.

\subsection{Horocycles}

We would now like to analyze how the different higher laminations representing a tropical point of $\Conf_n V$ are related to each other geometrically.

Let $v \in V$ be a vector. The group stabilizing the line given by the vector $v$ is a parabolic subgroup $P \subset SL_k$. It has a Langlands decomposition $P=MAN$.

Let $K \subset G$ be a maximal compact group. On the symmetric space $G/K$, one often considers orbits of the group $MN$. They are called the horocycles associated to the parabolic $P$. They are a torsor for the group $A$, because $P$ acts transitively on $G/K$. 

We will consider a non-archimedean analogue of this picture, where the symmetric space $G/K$ is replaced by the affine building. In the affine Grassmannian $\Gr$, one can consider various semi-infinite orbits. For us, the important ones will be the orbits of $M(\cK)A(\cO)N(\cK)$.

\begin{definition} For a parabolic subgroup $P \subset G$ with Langlands decomposition $P=MAN$, consider the orbits of $M(\cK)A(\cO)N(\cK)$ in $\Gr_G$. We will call the image of such an orbit in the affine building a \emph{horocycle for the parabolic $P$}.
\end{definition}

Now suppose that we have any finite number of higher laminations that represent the same point $y^t \in \Conf_n V(\Z^t)$. For simplicity, let us suppose that we have two higher laminations $\{(x_i,\lambda_i)\}$ and $\{(x'_i,\lambda'_i)\}$. By moving $y^t$ appropriately in its lineality space, we may assume that these two higher laminations are given by actual configurations $\{x_i\}$ and $\{x'_i\}$, for simplicity.

\begin{prop} Up to equivalence of higher laminations, we can arrange so that $y_i$ and $y'_i$ lie the the same horocycle.
\end{prop}

\begin{proof} In fact, we will show something slightly stronger. Let $x^t \in \Conf_n \A(\Z^t)$ be the tropical point corresponding to the configuration $\{x_i\}$, and let $x'^t \in \Conf_n \A(\Z^t)$ be the tropical point corresponding to the configuration $\{x'_i\}$.

Lift $y^t$ to a point $y \in \Conf_n V(\cK_{>0})$. Let $y$ give the configuration of vectors $v_1, \dots, v_n \in \cK^k$. By Proposition \ref{wellbehaved} and its corollaries, we have that there exist $x, x' \in \Conf_n \A(\cK_{>0})$ lifting $x^t, x'^t \in \Conf_n \A(\Z^t)$ such that $\pi(x)=\pi(x')=y$.

Then we can construct from $x$ a positive configuration of points in the affine Grassmannian as in Section \ref{construction}. Call this configuration $\{\tilde{x_i}\}$. Similarly, we can construct from $x'$ the configuration $\{\tilde{x'_i}\}$. By construction, $v_i$ is a valuation minimizing vector in both $\tilde{x_i}$ and $\tilde{x'_i}$. Therefore it is a generator for both $\tilde{x_i}$ and $\tilde{x'_i}$.

It is not hard to check that all lattices which contain a vector $v$ as a generator lie in a semi-infinite orbit for $M(\cK)A(\cO)N(\cK)$, where $P=MAN$ is the stabilizer of $v$. Therefore, for each $i$, $\tilde{x_i}$ and $\tilde{x'_i}$ both lie in the horocycle for the parabolic $P_{v_i}$, where $P_{v_i}$ is the stabilizer of $v_i$.

Clearly $\{\tilde{x_i}\}$ and $\{\tilde{x'_i}\}$ give configurations in the building which are equivalent to $\{x_i\}$ and $\{x'_i\}$, respectively.
\end{proof}

The proof above can be extended to any finite number of higher laminations. Now let $y^t \in \Conf_n V(\Z^t)$. We can consider all the higher laminations representing $y^t$. We can then consider the subset of those higher laminations that are non-virtual, i.e. consist of actual positive configurations of $n$ points in the affine building. This set may be empty, but by moving $y^t$ in its lineality space, we can guarantee that there are some. We then have the following:

\begin{theorem} All the non-virtual laminations realizing $y^t$ can be realized by positive configurations $\{x_i\}$, where each $x_i$ is allowed to vary over some subset of a horocycle for a parabolic $P_{v_i}$. This parabolic $P_{v_i}$ is constructed as in the previous proposition.
\end{theorem}

Let us remark that each $x_i$ varies over some strict subset of the horocycle for $v_i$. This is because $v_i$ not only needs to be a generator for $x_i$ (which tells us that it lies in the horocycle), but also needs to be a valuation minimizing vector with respect to the configuration.

In a sense, then, for a tropical point $y^t$, we can think of the higher laminations representing it as a \emph{configuration of horocycles in the building}. Finally, we are then justified in making the following definition.

\begin{definition}\label{horolam} Let $y^t \in \Conf_n V(\Z^t)$. The equivalence class of higher laminations given by the set $(\pi^t)^{-1}(y^t) \subset \Conf_n \A(\Z^t)$ is a horocycle lamination. By construction, horocycle laminations are in bijection with tropical points of $\Conf_n V$.
\end{definition}

\section{Distinguished representatives}

\subsection{Heuristics} We would now like to present a dual realization of the points of $\Conf_n V(\Z^t)$. This subsection is motivational for, but mostly independent of, Section 6.2 in which we state and prove our main theorem. With that Theorem established, Section 7 continues with the story line developed in this subsection. 

In the previous section, we used the fibration 
$$\pi: \Conf_n \A \twoheadrightarrow \Conf_n V$$
to represent tropical points in $\Conf_n V(\Z^t)$ as equivalence classes of tropical points in $\Conf_n \A(\Z^t)$. The map $\pi$ gives us a natural map on functions
\begin{equation} \label{inclusion}
\pi^*: \cO(\Conf_n V) \hookrightarrow \cO(\Conf_n \A).
\end{equation}
We will set out to describe how $\pi^*$ can be used to realize $\Conf_n V(\Z^t)$ as a subset of $\Conf_n \A(\Z^t)$.

As a preliminary step, we will need to define some auxiliary spaces. Recall that the principal flag variety $\A$ is the quotient $G/U$ for $G=\SL_k$. There is a related space $\A'=G'/U$, the principal flag variety associated to the adjoint group $G'=\PGL_k$, as well as a map
$$\phi: \A \twoheadrightarrow \A'.$$
This map is the quotient by the action of the center of $\SL_k$, which is given by the $k$-th roots of unity. Similarly, there is a map
$$\phi: V \twoheadrightarrow V'$$
also given by the quotient by the action of $k$-th roots of unity. 

\begin{rmk}
We can give a Lie theoretic definition of $V'$ that generalizes to other partial flag varieties and other semi-simple groups $G$. Let $G$ be simply connected, and let $G'$ be the adjoint group. For any parabolic $P \subset G$ there is a corresponding parabolic $P' \subset G'$. The flag varieties $G/P$ and $G/P'$ are naturally isomorphic. Let $P=MAN$ and $P'=MA'N$ be the Langlands decompositions. The factors $M$ and $N$ in each decomposition are isomorphic. However $A'$ is the quotient of $A$ by the center. The associated principal flag varieties are $G/MN$ and $G'/MN$. The latter is the quotient of the former by the center. Let us note that when $P$ is the stabilizer of a vector, the spaces $V$ and $G/MN$ are not exactly the same. The spaces differ only in codimension at least $2$, so the distinction plays no role for us.
\end{rmk}

We then have four different spaces, related as follows:
\[
\begin{tikzcd}
\Conf_n \A \arrow[d, twoheadrightarrow] \arrow[r, twoheadrightarrow] & \Conf_n \A' \arrow[d, twoheadrightarrow] \\
\Conf_n V \arrow[r, twoheadrightarrow] & \Conf_n V'
\end{tikzcd}
\]

Let us say a few words about the sets $\Conf_n \A'(\Zt)$ and $\Conf_n V'(\Zt)$. Note that the surjection $\A \twoheadrightarrow \A'$ is a $k : 1$ surjection. Therefore the map $\A^n \to (\A')^n$ is $k^n:1$. Quotienting by the diagonal action of roots of unity shows that the map $\Conf_n \A \twoheadrightarrow \Conf_n \A'$ is $k^{n-1}:1$. Pulling back gives an inclusion $\cO(\Conf_n \A') \hookrightarrow \cO(\Conf_n \A)$. The fact that there are fewer functions on $\Conf_n \A'$ means that it is \emph{easier} to be an integral tropical point. Thus there are inclusions
$$\Conf_n \A(\Zt) \subset \Conf_n \A'(\Zt),$$
$$\Conf_n V(\Zt) \subset \Conf_n V'(\Zt),$$
and in both cases the former has index $k^{n-1}$ in the latter.
Let us note that whereas points of $\Conf_n \A(\Zt)$ and $\Conf_n V(\Zt)$ are given by configurations of points in the affine building coming from the affine Grassmannian for $\SL_k$, points of $\Conf_n \A'(\Zt)$ and $\Conf_n V'(\Zt)$ are given by configurations of points in the affine building coming from the affine Grassmannian for $\PGL_k$. To summarize, we have relationships between various tropical spaces as in the following diagram:

\[
\begin{tikzcd}
\Conf_n \A(\Z^t) \arrow[d, twoheadrightarrow] \arrow[r, hook] & \Conf_n \A'(\Z^t) \arrow[d, twoheadrightarrow] \\
\Conf_n V(\Z^t)  \arrow[r, hook] & \Conf_n V'(\Z^t) 
\end{tikzcd}
\]

Now the duality conjectures state that there is a cone \eqref{eq:hiveinequalities} of tropical points
$$\Conf^+_n \A'(\Zt) \subset \Conf_n \A'(\Zt)$$
which parameterizes a basis of global functions in $\cO(\Conf_n \A)$.

Let us formulate a natural analogue of this statement for the space $\Conf_n V$. To a first approximation, the duality conjectures state that there should be a cone of tropical points
$$\Conf^+_n V'(\Zt) \subset \Conf_n V'(\Zt),$$
parameterizing a basis of global functions in $\cO(\Conf_n V)$. There are also analogous statements where the pairs $\A, \A'$ and $V, V'$ are interchanged, but for simplicity, we choose not to deal with them here.

This above statement needs to be modified slightly. First, let us denote by $\cO'(\Conf_n V)$ the algebra obtained from $\cO(\Conf_n V)$ by inverting the frozen varibables. (We note that the ring $\cO'(\Conf_n V)$ is the homogeneous coordinate ring of the {\it top positroid cell} for $\Gr(k,n)$). The $\X$-space associated to $\Conf_n V$ is $\Conf_n \mathbb{P}^{k-1}$, i.e. the configuration space of $n$-tuples in $\mathbb{P}^{k-1}$ (we elaborate on this in Section~8). Then the usual formulation of the duality conjectures is that the tropical points in $\Conf_n \mathbb{P}^{k-1}(\Zt)$ parameterize a basis for the functions in $\cO'(\Conf_n V)$ which are invariant under the action of $T^n$ on $\Conf_n V$, i.e., the functions which only depend on the lineality space of a point in $\Conf_n V$.

There are two steps to extending this statement. First, we have to analyze the map 
$$\Conf_n V' \rightarrow \Conf_n \mathbb{P}^{k-1}.$$ 
We will consider a related space $\Conf_n^* V' $ whose functions contain certain prescribed monomials in the frozen variables. We will postpone describing the space $\Conf_n^* V'$ for the moment, but let us mention that the spaces $\Conf_n^* V'$ and $\Conf_n V'$ have the same positive real points, i.e., $\Conf_n^* V'(\R_{>0}) \simeq \Conf_n V'(\R_{>0})$. The result will be that on the level of \emph{integral} tropical points, we have that
$$\Conf_n^* V'(\Zt) \subset \Conf_n V'(\Zt),$$
so that $\Conf_n^* V'(\Zt)$ can be thought of as a finite index sublattice of $\Conf_n V'(\Zt).$ To put it another way, the spaces $\Conf_n^* V'(\Qt)$ and $\Conf_n V'(\Qt)$ are isomorphic, but have different integral structures. Then the set $\Conf_n^* V'(\Zt)$ will parameterize a basis of functions in $\cO'(\Conf_n V)$. Now, if we ask that these functions extend across the divisors given by the frozen variables, then we will get a cone
$$\Conf_n^{*,+} V'(\Zt) \subset \Conf_n^* V'(\Zt)$$
which parameterizes a basis of global functions in $\cO(\Conf_n V)$.

Thus, we expect that the map in equation \eqref{inclusion} induces an inclusion
\begin{equation} \label{tropinclusion}
\Conf^{*,+}_n V'(\Zt) \subset \Conf^+_n \A'(\Zt).
\end{equation}

We then get a relationship between tropical spaces as follows:

\[
\begin{tikzcd}
\Conf_n^+ \A(\Z^t) \arrow[r, hook] & \Conf^+_n \A'(\Z^t)  \\
\Conf_n^{*,+} V(\Z^t) \arrow[u, hook] \arrow[r, hook] & \Conf_n^{*,+} V'(\Z^t) \arrow[u, hook]
\end{tikzcd}
\]

There is a similar diagram without the superscripts $+$. On the level of rational tropical points, $\Conf_n^* V'(\Qt) \simeq \Conf_n V'(\Qt)$ and $\Conf_n^* V(\Qt) \simeq \Conf_n V(\Qt)$, so that combining the previous two commutative diagrams, we have:

\[
\begin{tikzcd}
\Conf_n \A(\Q^t) \arrow[d, twoheadrightarrow, shift right=1ex] \arrow[r, hook] & \Conf_n \A'(\Q^t) \arrow[d, twoheadrightarrow, shift right=1ex] \\
\Conf_n V(\Q^t) \arrow[u, hook, shift right=1ex] \arrow[r, hook] & \Conf_n V'(\Q^t) \arrow[u, hook, shift right=1ex]
\end{tikzcd}
\]

Now let us recall some more about how $\Conf^+_n \A'(\Z^t)$ parameterizes functions in $\cO(\Conf_n \A)$. Suppose that $x^t \in \Conf^+_n \A'(\Z^t)$ is given by a configuration $x_1, x_2, \dots, x_n$. Let $d(x_{i-1}, x_i) = \lambda_i$, where indices are taken cyclically$\mod n$. Then $x^t $ will parameterize a function that lies in the invariant space 
$$[V_{\lambda_1}^* \otimes V_{\lambda_2}^* \otimes \cdots \otimes V_{\lambda_n}^*]^G.$$

On the other hand, the functions in $\cO(\Conf_n V)$ are given by invariants in the space
$$[V_{\lambda_1}^* \otimes V_{\lambda_2}^* \otimes \cdots \otimes V_{\lambda_n}^*]^G$$
where each $\lambda_i$ is a multiple of the fundamental weight $\omega_1$ for $SL_k$. Therefore, it is natural to expect that the inclusion in equation \eqref{tropinclusion} is given by the subset of configurations $x_1, x_2, \dots, x_n$ where for each $i$, $d(x_{i-1}, x_i)$ is a multiple of $\omega_1$.

\begin{rmk} For the remainder of this paper, we consider higher laminations for the group $PGL_k$, which correspond to configurations in the affine building for $PGL_k$. In the affine building for $PGL_k$, distances are given by coweights for $PGL_k$, which we will view as weights for $SL_k$. The fundamental weight $\omega_i$ for $SL_k$ is $(1, \dots, 1, 0, \dots, 0)$, where the vector contains $i$ $1$'s.
\end{rmk}

\subsection{Statement of main theorem and key lemma}
We can now formulate the main statement of this section.

\begin{theorem}\label{distinguished} There is a cone of points $\Conf^+_n V'(\Qt) \subset \Conf_n V'(\Qt)$ with the following properties:
\begin{enumerate}
\item Each $y^t \in \Conf^+_n V'(\Qt)$ is represented by a \emph{unique} higher lamination $\{x_i\}$ such that $\lambda_i := d(x_{i-1}, x_i)$ is a rational multiple of $\omega_1$. Moreover, $y^t \in \Conf_n V'(\Qt)$ is represented by such a higher lamination if and only if $y^t \in \Conf^+_n V'(\Qt)$.
\item Any point $y^t \in \Conf_n V'(\Qt)$ is in the same lineality space as some point $y'^t \in \Conf^+_n V'(\Qt)$.
\item There are explicit inequalities (cf.~\eqref{ineq}) which cut out the cone $\Conf^+_n V'(\Qt)$ inside $\Conf_n V'(\Qt)$.
\end{enumerate}
\end{theorem}

The goal of this section will be to give a proof of the above theorem.

\begin{rmk} For this section, because we are dealing with only rational tropical points of the space $\Conf^+_n V'$ (respectively $\Conf^+_n \A'$) it will not be important to distinguish between $V'$ and $V$ (respectively $\A'$ and $\A$). However, we will later deal with the integral points of these spaces, where we prefer to work with $V'$ and $\A'$ (though all statements can be appropriately modified to deal with $V$ and $\A$).
\end{rmk}

Let us first say a bit about why the above theorem is interesting. A point in $\Conf_n \A'(\Qt)$ can be described its tropical coordinates in some cluster chart. By the cluster fibration property, there is a cluster such that the coordinates on $\Conf_n V'$ are a strict subset of those on $\Conf_n \A'$. Theorem~\ref{distinguished} says that those configurations where each $d(x_{i-1}, x_i) = a_i \omega_1$ for some $a_i$ are determined by the tropical coordinates of this smaller subset, provided we also assume the hive inequalities hold. 
Thus the general strategy of our proof will be to analyze the interplay between the restriction that $d(x_{i-1}, x_i)= a_i \omega_1$ and the hive inequalities.

\medskip

Now we recall the hive inequalities. As in Section \ref{flags}, our first step is consider a configuration of three principal flags $A, B, C$, and let us consider the functions $f_{i_1i_2i_3}(A, B, C) = A^{i_1}B^{i_2}C^{i_3}$ where $i_1+i_2+i_3=k$. For simplicity--and because we will only be dealing with tropical points--throughout this section, we will notate the tropical functions as 
$A^{i_1}B^{i_2}C^{i_3}$ rather than $f^t_{i_1i_2i_3}(A, B, C)$ or $(A^{i_1}B^{i_2}C^{i_3})^t$. Then the hive inequalities state that the cone of points $\Conf^+_3 \A'(\Qt)$ (as well as $\Conf^+_3 \A'(\Zt)$) is cut out by the following three types of inequalities:
\begin{itemize}
\item $A^{i+1}B^{j+1}C^{k}+A^{i+1}B^{j}C^{k+1} \geq A^{i+2}B^{j}C^{k}+A^{i}B^{j+1}C^{k+1},$
\item $A^{i}B^{j+1}C^{k+1}+A^{i+1}B^{j+1}C^{k} \geq A^{i}B^{j+2}C^{k}+A^{i+1}B^{j}C^{k+1},$
\item $A^{i+1}B^{j}C^{k+1}+A^{i}B^{j+1}C^{k+1} \geq A^{i}B^{j}C^{k+2}+A^{i+1}B^{j+1}C^{k}.$
\end{itemize} 

These three sets of inequalities are associated with the flags $A, B, C$ respectively. By convention, we will always assume that the values at the boundary vertices (in this case, at $A,B,C$) are $0$, i.e. that $A^k=B^k=C^k=0$. If we associate the triples $(i_1,i_2,i_3)$ with the lattice points in a triangular array as in Figure~\ref{fig:Sl5seed}, then the three sets of inequalities above can be stated uniformly by requiring that, in each rhombus, the sum of the two obtuse angles is at least the sum of the two acute angles. 

A similar set of inequalities describes the cone $\Conf^+_n \A'(\Qt) \subset \Conf_n \A'(\Qt)$ when there are more than three flags. To write down these inequalities, choose any triangulation of the $n$-gon, and take the cluster associated to this triangulation. Then the hive inequalities from the various triangles cut out the desired cone. It is a fact that different triangulations give an equivalent set of inequalities. Moreover, each vertex of the $n$-gon has a subset of inequalities associated with it--namely the inequalities coming from each triangle that the vertex belongs to.

We can give another, more geometric interpretation of the hive inequalities, beginning again with the case of three flags. For each lattice point $(i_1,i_2,i_3)$ in the triangular array  we plot a point at a height of $A^{i_1}B^{i_2}C^{i_3}$ units above the lattice point. This is the graph of the tropical coordinates $A^{i_1}B^{i_2}C^{i_3}$ at the the inputs $(i_1,i_2,i_3)$. The convex hull says that the convex hull of these points should be modeled on Figure~\ref{fig:viewfromabove} (drawn in the case of $\SL_4$).

\begin{figure}[h]
\begin{tikzpicture}[scale=1]

  \node (x004) at (0,1.8) {};
  \node [above] at (x004) {$A$};

  \node (x103) at (-0.5,0.9) {};
  \node (x013) at (0.5,0.9) {};

  \node (x202) at (-1,0) {};
  \node (x112) at (0,0) {};
  \node (x022) at (1,0) {};

  \node (x301) at (-1.5,-0.9) {};
  \node (x211) at (-0.5,-0.9) {};
  \node (x121) at (0.5,-0.9) {};
  \node (x031) at (1.5,-0.9) {};

  \node [label=270:$B$] (x400) at (-2,-1.8) {};
  \node (x310) at (-1,-1.8) {};
  \node (x220) at (0,-1.8) {};
  \node (x130) at (1,-1.8) {};
  \node [label=270:$C$] (x040) at (2,-1.8) {};

  \draw [] (0,1.8) to (2,-1.8);
  \draw [] (-0.5,0.9) to (1,-1.8);
  \draw [] (-1,0) to (0,-1.8);
  \draw [] (-1.5,-0.9) to (-1,-1.8);

  \draw [] (2,-1.8) to (-2,-1.8);
  \draw [] (1.5,-0.9) to (-1.5,-0.9);
  \draw [] (1,0) to (-1,0);
  \draw [] (0.5,0.9) to (-0.5,0.9);

  \draw [] (-2,-1.8) to (0,1.8);
  \draw [] (-1,-1.8) to (0.5,0.9);
  \draw [] (0,-1.8) to (1,0);
  \draw [] (1,-1.8) to (1.5,-0.9);

\end{tikzpicture}
\caption{\label{fig:viewfromabove} Convex hull determined by hive inequalities for ${SL_4}$, in the generic case.}
\end{figure}
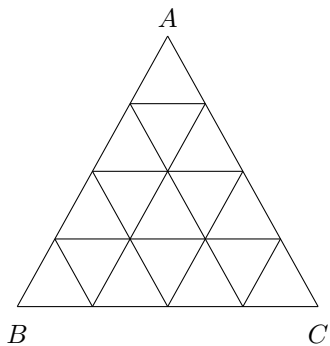

We recall our convention that the values at the vertices of the triangle are taken to be $0$. The picture in Figure~\ref{fig:viewfromabove} is the generic situation. When a given hive inequality happens to be an equality, it causes two adjacent triangles in the above diagram to become coplanar, and thus to become part of larger \emph{flat spaces} or what Knutson-Tao call \emph{puzzle pieces}. We will see examples of this now. 

Suppose that we have a positive configuration of three points $\{x_A, x_B, x_C\}$ where $d(x_B, x_{C}) = a \omega_1$. A straightforward calculation gives that $B^iC^{k-i} = \frac{ai}{k}$ as long as $0 < i < k$. We can interpet this as saying that in the corresponding convex hull, the points above the side $\overline{BC}$ lie on a line of slope $\frac{a}{k}$.   

The following Lemma \ref{lem:keyABC} will be important for us. It says that in the presence of the hive inequalities and the assumption $d(x_B, x_{C}) = a \omega_1$, the convex hull generically looks like Figure~\ref{fig:oneside}. Such a diagram is an example of a \emph{puzzle} in the sense Knutson-Tao. 
Each interior segment in Figure~\ref{fig:viewfromabove} corresponds to one of the hive inequalities; such a segment is the diagonal of a rhombus whose vertices are involved in the corresponding hive inequality. In a puzzle (cf. the ones in Figures \ref{fig:oneside}, \ref{fig:hulltwosides}, and \ref{fig:octagon}), the segments that {\it are} present correspond to inequalities that (could potentially) hold strictly, while missing segments correspond to inequalities that are actually equalities. These more general puzzles depict degenerate hives.

\begin{figure}[ht]
\begin{tikzpicture}[scale=1]

  \node (x004) at (0,1.8) {};
  \node [above] at (x004) {$A$};

  \node (x103) at (-0.5,0.9) {};
  \node (x013) at (0.5,0.9) {};

  \node (x202) at (-1,0) {};
  \node (x112) at (0,0) {};
  \node (x022) at (1,0) {};

  \node (x301) at (-1.5,-0.9) {};
  \node (x211) at (-0.5,-0.9) {};
  \node (x121) at (0.5,-0.9) {};
  \node (x031) at (1.5,-0.9) {};

  \node [label=270:$B$] (x400) at (-2,-1.8) {};
  \node (x310) at (-1,-1.8) {};
  \node (x220) at (0,-1.8) {};
  \node (x130) at (1,-1.8) {};
  \node [label=270:$C$] (x040) at (2,-1.8) {};

  \draw [] (0,1.8) to (2,-1.8);
  \draw [] (-0.5,0.9) to (0,0);
  \draw [] (-1,0) to (-0.5,-0.9);
  \draw [] (-1.5,-0.9) to (-1,-1.8);

  \draw [] (2,-1.8) to (-2,-1.8);
  \draw [] (1.5,-0.9) to (-1.5,-0.9);
  \draw [] (1,0) to (-1,0);
  \draw [] (0.5,0.9) to (-0.5,0.9);

  \draw [] (-2,-1.8) to (0,1.8);
  \draw [] (-1,-1.8) to (0.5,0.9);

\end{tikzpicture}
\caption{\label{fig:oneside} Generic convex hull when the bottom side satisfies $d(x_B, x_{C}) = a \omega_1$.}
\end{figure}
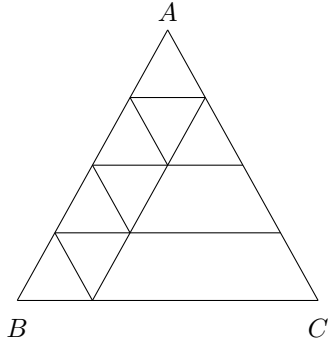

\begin{lemma}\label{lem:keyABC}
Let $\{x_A, x_B, x_C\}$ be a positive configuration of three points in the affine building where $d(x_B, x_C) = a \omega_1.$ Then the tropical coordinates obey
$$A^{i_1}B^{i_2}C^{i_3} = A^{i_1}C^{i_2+i_3} + \frac{ai_2}{k}$$
whenever $i_3 > 0$.
\end{lemma}

Put a different way, Lemma~\ref{lem:keyABC} says that each of the long horizontal segments in Figure~\ref{fig:oneside} is a line of slope $\frac{a}{k}$.  

\begin{proof} This is a straightforward consequence of the hive inequalities and the fact that $B^iC^{k-i} = \frac{ai}{k}$. For example, consider the three triangles in the bottom right fragment of Figure~\ref{fig:viewfromabove}: 
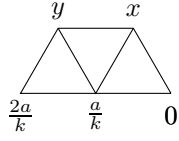
\begin{equation*}
\begin{tikzpicture}
\draw (2,0)--(0,0)--(.5,.866)--(1.5,.866)--(2,0);
\draw (.5,.866)--(1,0)--(1.5,.866);
\node at (2,-.3) {0};
\node at (1,-.3) {$\frac{a}{k}$};
\node at (0,-.3) {$\frac{2a}{k}$};
\node at (2,-.3) {0};
\node at (.5,1.1) {$y$};
\node at (1.5,1.1) {$x$};
\end{tikzpicture}.
\end{equation*}
These three triangles yield two hive inequalities $x+\frac{a}{k} \geq y$  and $y+\frac{a}{k} \geq x+\frac{2a}{k}$. Together, they imply $y = x+ \frac{a}{k}$. The rest of the argument proceeds in the same way. 
\end{proof}

We will make frequent use of the following minor extension of Lemma 6.4. 
 
\begin{lemma}\label{key}
Let $\{x_{A_1}, x_{A_2}, \dots, x_{A_m}, x_B, x_C\}$ be a positive configuration of $m+2$ points in the affine building where $d(x_B, x_C) = a \omega_1.$ Then the tropical coordinates obey
\begin{equation}\label{keyeqn}
A_1^{i_1} \cdots A_m^{i_m}B^{i_{m+1}}C^{i_{m+2}} = A_1^{i_1} \cdots A_m^{i_m}C^{i_{m+1}+i_{m+2}} + \frac{ai_{m+1}}{k}
\end{equation}
whenever $i_{m+2} > 0$.
\end{lemma}

\begin{proof} Take ordered bases for $x_{A_j}$ consisting of
$$v_{j1}, v_{j2}, \dots, v_{jk}.$$
Let $x_A$ be given by the lattice generated by
\begin{align*}
&v_{11}, \dots, v_{1i_1}\\
&v_{21}, \dots, v_{2i_2}\\
&\vdots\\
&v_{m-1,1}, \dots, v_{m-1,i_{m-1}} \\
&v_{m1}, \dots, v_{ml},
\end{align*}
where $l$ is taken so that $i_1+i_2+\cdots+i_{m-1}+l=k$. Then apply Lemma 6.4 to the configuration $\{x_A, x_B, x_C\}$. 
\end{proof}

Suppose we have a higher lamination given by (non-virtual) configuration $\{x_i\}$ such that the distance functions along the boundary are rational multiplies of the first fundamental coweight, i.e. $d(x_{i-1}, x_{i}) = a_i \omega_1$ for all $i$. We treat all indices cyclically $\mod n$. Using Lemma~\ref{key} repeatedly, we can establish a relationship between the $a_i$ and the tropical coordinates coming from the frozen variables for $\Conf_n V'$. 

\begin{align*}
0 &=(i+k-1)^k \\
&= (i+k-1)^1(i+k-2)^{k-1}-\frac{(k-1)a_{i+k-1}}{k} \\
&= (i+k-1)^1(i+k-2)^1(i+k-3)^{k-2}-\frac{(k-1)a_{i+k-1}}{k}-\frac{(k-2)a_{i+k-2}}{k} \\
&= \quad \quad \vdots \\
&= (i+k-1)^1(i+k-2)^1\cdots(i+1)^1(i)^1-\sum_{j=1}^{k-1} \frac{ja_{i+j}}{k}.
\end{align*}
In the above equations, a number with an exponent (e.g. $(i+k-1)^1$) stands for a boundary vertex being used in a tropical function, whereas a number without an exponent (e.g. $(k-1)$) is just the integer $k-1$. Thus the leftmost term in each expression is the value of a tropical function, while the other terms are evaluated using ordinary arithmetic.

The second line of the above equation follows from the first line by Lemma~\ref{key} applied to $B^{k-1}C^1$, where $B$ is the $(i+k-2)$th boundary vertex and $C$ is the $(i+k-1)$th. The remaining lines are similar.  

We denote $F_i := (i+k-1)^1(i+k-2)^1\cdots (i+1)^1(i)^1$. We have established
\begin{equation}\label{frozens}
F_i = \sum_{j=1}^{k-1} \frac{ja_{i+j}}{k}.
\end{equation}

The tropical function $F_i$ is the pullback $\pi^*(P^t(i,\dots,i+k-1))$ of the tropical Pl\"ucker coordinate $P^t(i,\dots,i+k-1)$. Moreover this Pl\"ucker coordinate is a frozen variable for $\tGr(k,n)$. Let us say a bit about notation here. It will be convenient to sometimes write, for example, $P^t(a, b, c, \dots)$ instead of $(a)^1(b)^1(c)^1\cdots$. We use the notations interchangeably. Moreover, for simplicity, we will also denote by $P^t$ the pullack $\pi^*(P^t)$. It should be clear at any given time on which space $P^t$ is a tropical function.

Now let us return to Equation \eqref{frozens}. This equation says that the frozen Pl\"ucker coordinates are determined by the distances $d(x_{i-1}, x_i) = a_i \omega_1$. The converse is true as well. If we know that a higher lamination satisfies $d(x_{i-1}, x_i) = a_i \omega_1$ for some $a_i$, and if we know the values of $F_i$ for all $i$, then we can determine the $a_i$. This amounts to inverting the matrix describing Equation \eqref{frozens}. 

\medskip

\begin{rmk} Although it is not important for us, let us explain why Equation \eqref{frozens} is invertible. Consider the vector space $\Q^n$. Let $T$ denote the map
$$T(a_1, a_2, \dots, a_n) = (a_n, a_1, a_2, \dots a_{n-1}).$$
Let us call the map $T$ the cyclic shift on $\Q^n$. Clearly $T^n=I$. Note that invertibility of Equation \eqref{frozens} is equivalent to the statement that the cyclic shifts of the vector 
$$v=(k-1, k-2, \dots, 3, 2, 1, 0, \dots, 0)$$ 
form a basis of $\Q^n$. 
This is equivalent to the cyclic shifts of 
\begin{equation}\label{weqn}
w = v-Tv=(k-1, -1, \dots, -1, 0, \dots, 0)
\end{equation}
forming a basis of the subspace $a_1+\cdots+a_n=0$.

Let $w = (a_1, a_2, \dots, a_n)$ where $a_1+\cdots+a_n=0$. Then the cyclic shifts of $w$ span the subspace $a_1+\cdots+a_n=0$ if and only if
$$\sum_{i=1}^n a_i \zeta^i \neq 0$$
for any $n$-th root of unity $\zeta \neq 1$. For our particular $w$ in Equation \ref{weqn}, this holds by the triangle inequality.

\end{rmk}
\medskip

Now we will use calculations similar to those giving Equation \eqref{frozens} to derive the inequalities which cut out the cone $\Conf^+_n V'(\Qt) \subset \Conf_n V'(\Qt)$. Suppose we have a triangle formed by a positive configuration of points $x_1, x_2, x_3 \in \Gr$, and suppose furthermore we know that $d(x_1,x_2)$ and $d(x_2,x_3)$ are rational multiples of $\omega_1$. Using Lemma~\ref{lem:keyABC} for these two sides, it follows that the convex hull generically looks like Figure~\ref{fig:hulltwosides}.

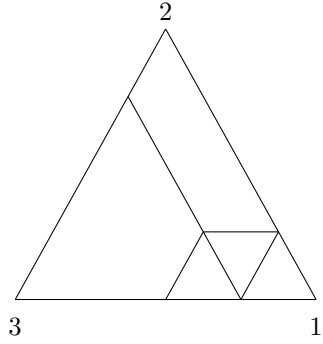
\begin{figure}[ht]
\begin{tikzpicture}[scale=1]

  \node (x004) at (0,1.8) {};
  \node [above] at (x004) {$2$};

  \node (x103) at (-0.5,0.9) {};
  \node (x013) at (0.5,0.9) {};

  \node (x202) at (-1,0) {};
  \node (x112) at (0,0) {};
  \node (x022) at (1,0) {};

  \node (x301) at (-1.5,-0.9) {};
  \node (x211) at (-0.5,-0.9) {};
  \node (x121) at (0.5,-0.9) {};
  \node (x031) at (1.5,-0.9) {};

  \node [label=270:$3$] (x400) at (-2,-1.8) {};
  \node (x310) at (-1,-1.8) {};
  \node (x220) at (0,-1.8) {};
  \node (x130) at (1,-1.8) {};
  \node [label=270:$1$] (x040) at (2,-1.8) {};

  \draw [] (0,1.8) to (2,-1.8);
  \draw [] (-0.5,0.9) to (1,-1.8);

  \draw [] (0.5,-0.9) to (1.5,-0.9);
   \draw [] (2,-1.8) to (-2,-1.8);

  \draw [] (0,-1.8) to (0.5,-0.9);
  \draw [] (1,-1.8) to (1.5,-0.9);
  \draw [] (0,1.8) to (-2,-1.8);


\end{tikzpicture}
\caption{\label{fig:hulltwosides} Generic convex hull when the left and ride sides satisfy $d(x_{i-1}, x_{i}) = a_i \omega_1$. }
\end{figure}

We see that the only tropical coordinate which is not yet determined is $1^{k-1}3$. Now let us consider all the hive inequalities associated to the point $2$. These are the inequalities associated to vertical rhombii in Figure~\ref{fig:viewfromabove}, i.e. those rhombii with horizontal diagonal. All of these inequalities are forced to hold except for one, corresponding to the unique horizontal line in the interior of Figure~\ref{fig:hulltwosides}:
\begin{align*}
1^{k-2}2^13^1 + 1^{k-1}2^1 &\geq 1^{k-1}3^1 + 1^{k-2}2^2, \text{ or equivalently } \\
1^{k-2}2^13^1 &\geq 1^{k-1}3^1 - \frac{a_2}{k}.
\end{align*}

We can relate both $1^{k-2}2^13^1$ and $1^{k-1}3^1$ to tropical Pl\"ucker coordinates using the same iterative procedure that we used to establish \eqref{frozens}.The preceding inequality becomes:  
\begin{equation}\label{computePluckers}
P^t(n-k+4,\dots,3)- \sum_{j=1}^{k-3} \frac{ja_{n-k+4+j}}{k}  \geq P^t(n-k+3,\dots,1,3)-\sum_{j=1}^{k-2} \frac{ja_{n-k+3+j}}{k} - \frac{a_2}{k}
\end{equation}
which, after reindexing sums and canceling, simplifies to 
\begin{equation}\label{eq:almostlaststep}
P^t(n-k+4,\dots,3)+\sum_{j=1}^{k-1} \frac{a_{n-k+3+j}}{k} \geq P^t(n-k+3,\dots,1,3).
\end{equation}
Using \eqref{frozens}, we have that $P^t(n-k+3,\dots,2)-P^t(n-k+4,\dots,3) = \sum_{j=1}^{k-1} \frac{a_{n-k+3+j}}{k}-\frac{(k-1)a_3}{k}$. So \eqref{eq:almostlaststep} simplifies to 
$$ P^t(n-k+3,\dots,2)+ \frac{k-1}{k}a_3 \geq P^t(n-k+3,\dots,1,3).$$

Let us denote by define $F^+_i =P^t(i,i+1,\dots,i+k-2,i+k)= i^1\cdots(i+k-2)^1(i+k)^1$. Generalizing the preceding calculation (by rotating), we obtain





\begin{equation}\label{ineq}
F_i + \frac{k-1}{k}a_{i+k} \geq F_i^+ \text{ for all $i$.}
\end{equation}
Note the resemblance to the inequalities found in \cite{RW} (Equation (5.1)). Now, observe that each $a_i$ can be written as a (somewhat complicated) linear combination of the functions $F_i$ by inverting Equation~\eqref{frozens}. Thus Equation~\eqref{ineq} defines a set of inequalities on the tropical Pl\"ucker coordinates for $\Conf_n V'(\Qt)$. This is the cone that we will call $\Conf^+_n V'(\Qt)$.

\subsection{Construction and proof}

We will now proceed to give a proof of Theorem \ref{distinguished}. The heart of the proof is the following statement, the parts of which we will prove in tandem with Theorem \ref{distinguished}.

\begin{prop}\label{uniquelift}
Let $y^t \in \Conf_n V'(\Qt)$. Evaluating the frozen variables $F_i$ at $y^t$ determines the boundary values $a_i$ using \ref{frozens}. Then there is a unique $x^t \in \Conf_n \A'(\Qt)$ satisfying Equation~\eqref{keyeqn} (for these $a_i$) such that $\pi^t(x^t)=y^t$.
\end{prop}

\begin{rmk} The $a_i$ are given by a linear combination of $F_i(y^t)$, and they may be negative for a general point $y^t \in \Conf_n V'(\Qt)$. It will be a consequence of the proof of Theorem \ref{distinguished} that if $y^t \in \Conf^+_n V'(\Qt)$, we will have that $a_i \geq 0$. (The converse of this statement is not true.)
\end{rmk}

We will illustrate the proof in detail in the case of $\Gr(4,8)$. The general case will be clear from this special case. 

We begin with the case that $y^t=\pi(x^t)$ where $x^t$ is a higher lamination given by a configuration $\{x_i\}$ such that each $d(x_{i-1}, x_{i})$ is a rational multiple of $\omega_1$. Since $x^t$ is an actual configuration, it is in the positive cone $x^t \in \Conf^+_8 \A'(\Qt)$. As usual, we number the flags in our configuration of flags $1, 2, \dots, 8$ and place them at the vertices of an octagon. We triangulate the octagon as in Figure~\ref{fig:octagon}. 

\begin{figure}
\begin{tikzpicture}[scale=1]

  \node (3) at (0,1.4) {};
  \node [above] at (3) {$3$};

  \node (4) at (-1,1) {};
  \node [above, left] at (4) {$4$};

  \node (5) at (-1.4,0) {};
  \node [left] at (5) {$5$};

  \node (6) at (-1,-1) {};
  \node [below, left] at (6) {$6$};

  \node (7) at (0,-1.4) {};
  \node [below] at (7) {$7$};

  \node (8) at (1,-1) {};
  \node [below, right] at (8) {$8$};

  \node (1) at (1.4,0) {};
  \node [right] at (1) {$1$};

  \node (2) at (1,1) {};
  \node [above, right] at (2) {$2$};

  \draw [] (0,1.4) to (-1,1);
  \draw [] (-1,1) to (-1.4,0);
  \draw [] (-1.4,0) to (-1,-1);
  \draw [] (-1,-1) to (0,-1.4);
  \draw [] (0,-1.4) to (1,-1);
  \draw [] (1,-1) to (1.4,0);
  \draw [] (1.4,0) to (1,1);
  \draw [] (1,1) to (0,1.4);

  \draw [] (0,1.4) to (0,-1.4);
  \draw [] (0,1.4) to (-1,-1);
  \draw [] (0,1.4) to (-1.4,0);
  \draw [] (0,1.4) to (1,-1);
  \draw [] (0,1.4) to (1.4,0);


\end{tikzpicture}
\caption{\label{fig:octagon}}
\end{figure}

On each of the resulting six triangles, we have the associated cluster coordinates on $\Conf^+_8 \A'$, placed at the lattice points of a triangular array, as usual. As in Figures~\ref{fig:oneside},\ref{fig:hulltwosides}, the hive inequalities force the convex hulls of these coordinates to be as pictured below. In the diagrams below, the bold lines describe segments or regions where all the coordinates are determined by some Pl\"ucker coordinates in combination with Equation~\eqref{keyeqn}:

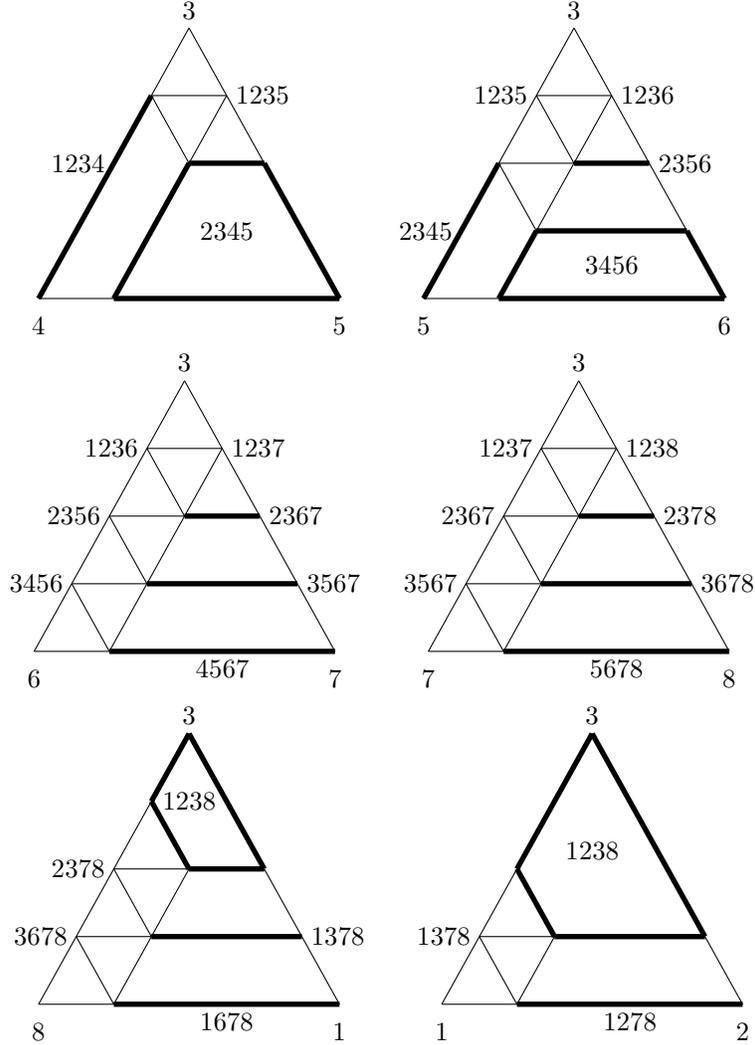
\begin{figure}
\begin{tikzpicture}[scale=1]

  \node (x004) at (0,1.8) {};
  \node [above] at (x004) {$3$};

  \node (x103) at (-0.5,0.9) {};
  \node (x013) at (0.5,0.9) {};
  \node [right] at (x013) {$1235$};

  \node (x202) at (-1,0) {};
  \node [left] at (x202) {$1234$};
  \node (x112) at (0,0) {};
  \node (x022) at (1,0) {};

  \node (x301) at (-1.5,-0.9) {};
  \node (x211) at (-0.5,-0.9) {};
  \node (x121) at (0.5,-0.9) {$2345$};
  \node (x031) at (1.5,-0.9) {};

  \node [label=270:$4$] (x400) at (-2,-1.8) {};
  \node (x310) at (-1,-1.8) {};
  \node (x220) at (0,-1.8) {};
  \node (x130) at (1,-1.8) {};
  \node [label=270:$5$] (x040) at (2,-1.8) {};

  \draw [] (-2,-1.8) to (2,-1.8);
  \draw [] (-0.5,0.9) to (0.5,0.9);
  \draw [line width=2pt] (0,0) to (1,0);
  \draw [line width=2pt] (-1,-1.8) to (2,-1.8);

  \draw [] (0,1.8) to (-2,-1.8);
  \draw [line width=2pt] (-0.5,0.9) to (-2,-1.8);
  \draw [] (0.5,0.9) to (-1,-1.8);
  \draw [line width=2pt] (0,0) to (-1,-1.8);

  \draw [] (0,1.8) to (2,-1.8);
  \draw [] (-0.5,0.9) to (0,0);
  \draw [line width=2pt] (1,0) to (2,-1.8);

\end{tikzpicture} \quad
\begin{tikzpicture}[scale=1]

  \node (x004) at (0,1.8) {};
  \node [above] at (x004) {$3$};

  \node (x103) at (-0.5,0.9) {};
  \node [left] at (x103) {$1235$};
  \node (x013) at (0.5,0.9) {};
  \node [right] at (x013) {$1236$};

  \node (x202) at (-1,0) {};
  \node (x112) at (0,0) {};
  \node (x022) at (1,0) {};
  \node [right] at (x022) {$2356$};

  \node (x301) at (-1.5,-0.9) {};
  \node [left] at (x301) {$2345$};
  \node (x211) at (-0.5,-0.9) {};
  \node (x121) at (0.5,-0.9) {};
  \node (3456) at (0.5, -1.35) {$3456$};
  \node (x031) at (1.5,-0.9) {};

  \node [label=270:$5$] (x400) at (-2,-1.8) {};
  \node (x310) at (-1,-1.8) {};
  \node (x220) at (0,-1.8) {};
  \node (x130) at (1,-1.8) {};
  \node [label=270:$6$] (x040) at (2,-1.8) {};

  \draw [] (-2,-1.8) to (2,-1.8);
  \draw [] (-0.5,-0.9) to (1.5,-0.9);
  \draw [] (-1,0) to (1,0);
  \draw [] (-0.5,0.9) to (0.5,0.9);
  \draw [line width=2pt] (0,0) to (1,0);
  \draw [line width=2pt] (-0.5,-0.9) to (1.5,-0.9);
  \draw [line width=2pt] (-1,-1.8) to (2,-1.8);

  \draw [] (-2,-1.8) to (0,1.8);
  \draw [line width=2pt] (-2,-1.8) to (-1,0);
  \draw [] (-1,-1.8) to (0.5,0.9);
  \draw [line width=2pt] (-1,-1.8) to (-0.5,-0.9);

  \draw [] (0,1.8) to (2,-1.8);
  \draw [] (-0.5,0.9) to (0,0);
  \draw [] (-1,0) to (-0.5,-0.9);
  \draw [line width=2pt] (1.5,-0.9) to (2,-1.8);

\end{tikzpicture}

\begin{tikzpicture}[scale=1]

  \node (x004) at (0,1.8) {};
  \node [above] at (x004) {$3$};

  \node (x103) at (-0.5,0.9) {};
  \node [left] at (x103) {$1236$};
  \node (x013) at (0.5,0.9) {};
  \node [right] at (x013) {$1237$};

  \node (x202) at (-1,0) {};
  \node [left] at (x202) {$2356$};
  \node (x112) at (0,0) {};
  \node (x022) at (1,0) {};
  \node [right] at (x022) {$2367$};

  \node (x301) at (-1.5,-0.9) {};
  \node [left] at (x301) {$3456$};
  \node (x211) at (-0.5,-0.9) {};
  \node (x121) at (0.5,-0.9) {};
  \node (x031) at (1.5,-0.9) {};
  \node [right] at (x031) {$3567$};

  \node [label=270:$6$] (x400) at (-2,-1.8) {};
  \node (x310) at (-1,-1.8) {};
  \node (x220) at (0,-1.8) {};
  \node (4567) at (0.5, -1.8) {};
  \node [below] at (4567) {$4567$};
  \node (x130) at (1,-1.8) {};
  \node [label=270:$7$] (x040) at (2,-1.8) {};

  \draw [] (-2,-1.8) to (2,-1.8);
  \draw [] (-1.5,-0.9) to (1.5,-0.9);
  \draw [] (-1,0) to (1,0);
  \draw [] (-0.5,0.9) to (0.5,0.9);
  \draw [line width=2pt] (0,0) to (1,0);
  \draw [line width=2pt] (-0.5,-0.9) to (1.5,-0.9);
  \draw [line width=2pt] (-1,-1.8) to (2,-1.8);

  \draw [] (-2,-1.8) to (0,1.8);
  \draw [] (-1,-1.8) to (0.5,0.9);

  \draw [] (0,1.8) to (2,-1.8);
  \draw [] (-0.5,0.9) to (0,0);
  \draw [] (-1,0) to (-0.5,-0.9);
  \draw [] (-1.5,-0.9) to (-1,-1.8);

\end{tikzpicture}\quad
\begin{tikzpicture}[scale=1]

  \node (x004) at (0,1.8) {};
  \node [above] at (x004) {$3$};

  \node (x103) at (-0.5,0.9) {};
  \node [left] at (x103) {$1237$};
  \node (x013) at (0.5,0.9) {};
  \node [right] at (x013) {$1238$};

  \node (x202) at (-1,0) {};
  \node [left] at (x202) {$2367$};
  \node (x112) at (0,0) {};
  \node (x022) at (1,0) {};
  \node [right] at (x022) {$2378$};

  \node (x301) at (-1.5,-0.9) {};
  \node [left] at (x301) {$3567$};
  \node (x211) at (-0.5,-0.9) {};
  \node (x121) at (0.5,-0.9) {};
  \node (x031) at (1.5,-0.9) {};
  \node [right] at (x031) {$3678$};

  \node [label=270:$7$] (x400) at (-2,-1.8) {};
  \node (x310) at (-1,-1.8) {};
  \node (x220) at (0,-1.8) {};
  \node (5678) at (0.5, -1.8) {};
  \node [below] at (5678) {$5678$};
  \node (x130) at (1,-1.8) {};
  \node [label=270:$8$] (x040) at (2,-1.8) {};

  \draw [] (-2,-1.8) to (2,-1.8);
  \draw [] (-1.5,-0.9) to (1.5,-0.9);
  \draw [] (-1,0) to (1,0);
  \draw [] (-0.5,0.9) to (0.5,0.9);
  \draw [line width=2pt] (0,0) to (1,0);
  \draw [line width=2pt] (-0.5,-0.9) to (1.5,-0.9);
  \draw [line width=2pt] (-1,-1.8) to (2,-1.8);

  \draw [] (-2,-1.8) to (0,1.8);
  \draw [] (-1,-1.8) to (0.5,0.9);

  \draw [] (0,1.8) to (2,-1.8);
  \draw [] (-0.5,0.9) to (0,0);
  \draw [] (-1,0) to (-0.5,-0.9);
  \draw [] (-1.5,-0.9) to (-1,-1.8);

\end{tikzpicture}

\begin{tikzpicture}[scale=1]

  \node (x004) at (0,1.8) {};
  \node [above] at (x004) {$3$};

  \node (x103) at (-0.5,0.9) {};
  \node (x013) at (0.5,0.9) {};
  \node (1238) at (0,0.9) {$1238$};

  \node (x202) at (-1,0) {};
  \node [left] at (x202) {$2378$};
  \node (x112) at (0,0) {};
  \node (x022) at (1,0) {};

  \node (x301) at (-1.5,-0.9) {};
  \node [left] at (x301) {$3678$};
  \node (x211) at (-0.5,-0.9) {};
  \node (x121) at (0.5,-0.9) {};
  \node (x031) at (1.5,-0.9) {};
  \node [right] at (x031) {$1378$};

  \node [label=270:$8$] (x400) at (-2,-1.8) {};
  \node (x310) at (-1,-1.8) {};
  \node (x220) at (0,-1.8) {};
  \node (1678) at (0.5, -1.8) {};
  \node [below] at (1678) {$1678$};
  \node (x130) at (1,-1.8) {};
  \node [label=270:$1$] (x040) at (2,-1.8) {};

  \draw [] (-2,-1.8) to (2,-1.8);
  \draw [] (-1.5,-0.9) to (1.5,-0.9);
  \draw [] (-1,0) to (1,0);
  \draw [line width=2pt] (0,0) to (1,0);
  \draw [line width=2pt] (-0.5,-0.9) to (1.5,-0.9);
  \draw [line width=2pt] (-1,-1.8) to (2,-1.8);

  \draw [] (-2,-1.8) to (0,1.8);
  \draw [] (-1,-1.8) to (0,0);
  \draw [line width=2pt] (-0.5,0.9) to (0,1.8);

  \draw [] (0,1.8) to (2,-1.8);
  \draw [line width=2pt] (0,1.8) to (1,0);
  \draw [] (-0.5,0.9) to (0,0);
  \draw [line width=2pt] (-0.5,0.9) to (0,0);
  \draw [] (-1,0) to (-0.5,-0.9);
  \draw [] (-1.5,-0.9) to (-1,-1.8);

\end{tikzpicture} \quad
\begin{tikzpicture}[scale=1]

  \node (x004) at (0,1.8) {};
  \node [above] at (x004) {$3$};

  \node (x103) at (-0.5,0.9) {};
  \node (x013) at (0.5,0.9) {};

  \node (x202) at (-1,0) {};
  \node (x112) at (0,0) {};
  \node [above] at (x112) {$1238$};
  \node (x022) at (1,0) {};

  \node (x301) at (-1.5,-0.9) {};
  \node [left] at (x301) {$1378$};
  \node (x211) at (-0.5,-0.9) {};
  \node (x121) at (0.5,-0.9) {};
  \node (x031) at (1.5,-0.9) {};

  \node [label=270:$1$] (x400) at (-2,-1.8) {};
  \node (x310) at (-1,-1.8) {};
  \node (x220) at (0,-1.8) {};
  \node (1278) at (0.5, -1.8) {};
  \node [below] at (1278) {$1278$};
  \node (x130) at (1,-1.8) {};
  \node [label=270:$2$] (x040) at (2,-1.8) {};

  \draw [] (-2,-1.8) to (2,-1.8);
  \draw [] (-1.5,-0.9) to (1.5,-0.9);
  \draw [line width=2pt] (-0.5,-0.9) to (1.5,-0.9);
  \draw [line width=2pt] (-1,-1.8) to (2,-1.8);

  \draw [] (-2,-1.8) to (0,1.8);
  \draw [] (-1,-1.8) to (-0.5,-0.9);
  \draw [line width=2pt] (-1,0) to (0,1.8);

  \draw [] (0,1.8) to (2,-1.8);
  \draw [line width=2pt] (0,1.8) to (1.5,-0.9);
  \draw [] (-1,0) to (-0.5,-0.9);
  \draw [line width=2pt] (-1,0) to (-0.5,-0.9);
  \draw [] (-1.5,-0.9) to (-1,-1.8);

\end{tikzpicture}
\caption{How the coordinates of $x^t \in \Conf_8 \A'$ are determined by the coordinates of $y^t \in \Conf_8 V'(\Qt)$ \label{fig:xtyt}}
\end{figure}




For example, in triangle $356$, we have that 
\begin{itemize}
\item $3^35^1$ is determined by $P^t(1235)$;
\item $3^36^1$ is determined by $P^t(1236)$;
\item $3^25^2$ and $35^3$ are determined by $P^t(2345)$;
\item $3^256$ and $3^26^2$ are determined by $P^t(2356)$;
\item $35^26$, $356^2$, $36^3$, $5^36, 5^26^2$ and $56^3$ are all determined by $P^t(3456)$.
\end{itemize}
Verifying these statements requires repeated use of the iterative process via Lemma~\ref{key}, similar to the calculations for \eqref{frozens} and \eqref{computePluckers}.

We see from this analysis that all the coordinates in this cluster for $\Conf_8 \A'$ are determined by Equation \eqref{keyeqn} and Pl\"ucker coordinates. In fact, we see more: the Pl\"ucker coordinates that appear are exactly those occurring in our standard cluster for $\Gr(4,8)$ in Figure 1. Thus we have the following:

\begin{prop} Suppose that $x^t \in \Conf^+_n \A'(\Qt)$ corresponds to a higher lamination $\{x_i\}$ such that $d(x_{i-1}, x_i) = a_i \omega_1$ for some $a_i \in \Q$. Then the coordinates of $x^t$ are completely determined by the coordinates of $y^t=\pi(x^t)$.
\end{prop}

Now, the coordinates of $x^t$ automatically satisfy Equation \eqref{keyeqn}: because $x^t$ is represented by a higher lamination $\{x_i\}$ such that $d(x_{i-1}, x_i)=a_i\omega_1$, Equation \eqref{keyeqn} holds by Lemma \ref{key}. Thus the proposition above gives us Theorem~\ref{uniquelift} for any $y^t$ such that $y^t=\pi(x^t)$ and $x^t$ corresponds to a higher lamination $\{x_i\}$ such that $d(x_{i-1}, x_i) = a_i \omega_1$.

\par
We now deal with the general case. Fix some $y^t.$ Then we can still impose Equation~\eqref{keyeqn} to construct a point $x^t \in \Conf_n \A'(\Qt)$. Just as in the previous case, all coordinates of $x^t$ are determined by those of $y^t$ after we impose Equation~\eqref{keyeqn}. The values of $F_i(y^t)$ determine the values of $a_i$, which in combination with the Pl\"ucker coordinates of $y^t$ determine the coordinates of $x^t$. In Figure 8, all the coordinates along a bold segment or in a bold-outlined region are determined by the labelled Pl\"ucker coordinate. Thus, we have uniqueness of $x^t$.

Let us point out that the convex hull of the coordinates of $x^t$ will \emph{not} look like Figure 8 in general. The coordinates along a bold segment in a bold-outlined region are still collinear or coplanar (again, using Equation~\eqref{keyeqn}), respectively, but the hive inequalities may not hold. $x^t \in \Conf^+_8 \A'(\Qt)$ satisfies the hive inequalities if and only if when one plots the coordinates of $x^t$, the convex hull looks like Figure 8.

Although if $x^t$ exists, it is unique, the existence of $x^t$ is a bit less clear--for example, it is not clear that the $x^t$ we have constructed is independent of triangulation. In other words, our construction of $x^t$ may not be well defined, because, in principle, imposing all instances of Equation \eqref{keyeqn} could lead to a contradiction. The coordinates of $x^t$ as above are only determined by some subset of instances of Equation \eqref{keyeqn}. For now, let us call $x^t$ the point of $x^t \in \Conf^+_8 \A'(\Qt)$ that is determined using the procedure given above, and say that $x^t$ is the \emph{lift} of $y^t$.

We would like to show that $x^t$ is independent of any choices. It will be sufficient to show that for the $x^t$ we have constructed, Equation \eqref{keyeqn} always holds. This was easy in the case that $x^t$ satisfied the hive inequalities. The more general case requires some argument.

Our strategy will be as follows: although for a general $y^t \in \Conf_8 V'(\Qt)$, the lift $x^t \in \Conf_8 \A'(\Qt)$ may not satisfy the hive inequalities, there is a point $y'^t \in \Conf^+_8 V'(\Qt)$ in the same lineality space as $y^t$ such that its lift $x'^t$ lies in $\Conf^+_8 \A'(\Qt)$ and thus satisfies the hive inequalities. An analysis of lineality spaces will then be enough to conclude that $x^t$ satisfies Equation \eqref{keyeqn}.

We now need to analyze the lineality space of $y^t$. For the moment we work in the general case of $\Gr(k,n)$. We begin by observing that there is a unique higher lamination $\{x_i\}$ such that $d(x_{i-1}, x_i) = \omega_1$ for $i=1, 2, \dots, k$ and $d(x_{i-1}, x_i) = 0$ otherwise. Let $l_1 \in \Conf^+_n \A'(\Qt)$ correspond to this higher lamination. We have the following lemma:

\begin{lemma} Consider the Pl\"ucker coordinate $A = A_1\cdots A_k$. Then 
$$A_1\cdots A_k (l_1) = \sum_{A_i < k} \frac{k-A_i}{k}.$$
\end{lemma}

\begin{proof} Our assumption on $\{x_i\}$ says that $a_1 = a_2 = \cdots a_k = 1$ while all other $a_j = 0$. Note that when $A_1,\dots,A_k = 1,\dots,k$, then $A$ is the frozen Pl\"ucker coordinate $F_1$.  Evaluating $F_1$ on $l_1$ using \eqref{frozens}, we get $\frac{1+2+\cdots+(k-1)}{k}$ which agrees with the claimed formula. 

Now consider an arbitrary Pl\:ucker coordinate $A$, with $A_1, \dots, A_i <k$ and $A_{i+1}, \dots, A_k \geq k$. Our assumption on $\{x_i\}$ implies that $x_k=x_{k+1}=\cdots = x_n$. So we can simplify $A_1^1\cdots A_k^1 = A_1^1 \cdots A_i^1 k^{k-i}.$ By judicious use of Lemma~\ref{key}, we can move the exponent $k-i$ leftward, decreasing this exponent at each step, and filling in any gaps in the numbers $A_1$ through $A_i$ as we go. We eventually arrive at $1^1 \cdots k^1$. To complete the proof, we need to analyze the multiples of $\frac{-1}{k}$ we absorb each time we apply Lemma~\ref{key}. We make the following observation. Suppose we replace a Pl\"ucker coordinate $A$ with a new Pl\"ucker coordinate $A' = A_1',\dots,A_k'$ by augmenting a single entry: $A_j' = A_j+1$ for some $j \leq i$. Then $A'_1 \cdots A'_i k^{k-i} = A_1 \cdots A_i k^{k-1}-\frac{1}{k}$. This is clear by comparing what happens as we move the exponents leftward. From this, and the base case $F_1(l_1) = \frac{1 + \dots k-1}{k}$, the result follows. 
\end{proof}

Let us rephrase this in a way that will be useful for us: if we take the point $0 \in \Conf_n V'(\Qt)$ and act by the element $\frac{1}{k}(k-1,\dots, 2, 1, 0, \dots, 0) \in T^n(\Qt)$, then we get the point $\pi(l_1) \in \Conf_n V'(\Qt)$. Now let $l_i \in \Conf^+_n \A'(\Qt)$ correspond to the unique higher lamination $\{x_i\}$ such that $d(x_{i-1}, x_i) = \omega_1$ for $i=i, 2, \dots, i+k-1$ and $d(x_{i-1}, x_i) = 0$ otherwise. 

\begin{cor}{\label{lineality}} Let $y^t \in \Conf_n V'(\Qt)$, and let $x^t \in \Conf_n \A'(\Qt)$ be its lift. Then the result of acting by 
$$\frac{1}{k}(k-1,\dots, 2, 1, 0, \dots, 0) \in T^n(\Qt)$$
(respectively, by an appropriate cyclic shift of $\frac{1}{k}(k-1,\dots, 2, 1, 0, \dots, 0)$) on $y^t$ is the same as adding the coordinates of $\pi(l_1)$ (resp., $\pi(l_i)$) to the coordinates of $y^t$. Call the resulting point $y^t + \pi(l_1)$ (resp., $y^t + \pi(l_i)$).

Moreover, the lift of $y^t + \pi(l_1)$ (resp., $y^t + \pi(l_i)$) comes from adding the coordinates of $l_1$ (respectively, $l_i$) to the coordinates of $x^t$. We will call the resulting tropical point $x^t + l_1$ (resp., $x^t + l_i$).
\end{cor}

The last statement of the above corollary follows from the linearity of the construction of $x^t$ from $y^t$.

The linear combinations $y^t + \sum c_i \pi(l_i)$ span the lineality space of $y^t$. Thus we have a different viewpoint on the action of the lineality space $T^n(\Qt)$: it is the space spanned by the $\pi(l_i)$.

\begin{rmk} Let us give some motivation for considering the space spanned by the $l_i$. Using the duality conjectures, the higher laminations $l_i$ correspond to the functions $F_i$. The corollary above states that moving around in the lineality space amounts to, on the dual side, multiplying by frozen variables. This statement is well known for the space $\Conf_n \A(\Qt)$, and the corollary is the analogous statement for $\Conf_n V'(\Qt)$. We expect a similar statement should hold more generally, for example for all positroid cells.
\end{rmk}

It will be more convenient to analyze the lineality space using the higher laminations $l_i$. We calculated the value of every Pl\"ucker coordinate on $l_i$, so we know the coordinates of $l_i$ in any cluster. Let us now return to the case of $\Gr(4,8)$ and look at the coordinates of the $l_i$ using the triangulation from Figure~\ref{fig:octagon}. The convex hulls of $l_1, l_2, \dots, l_8$ look as follows:

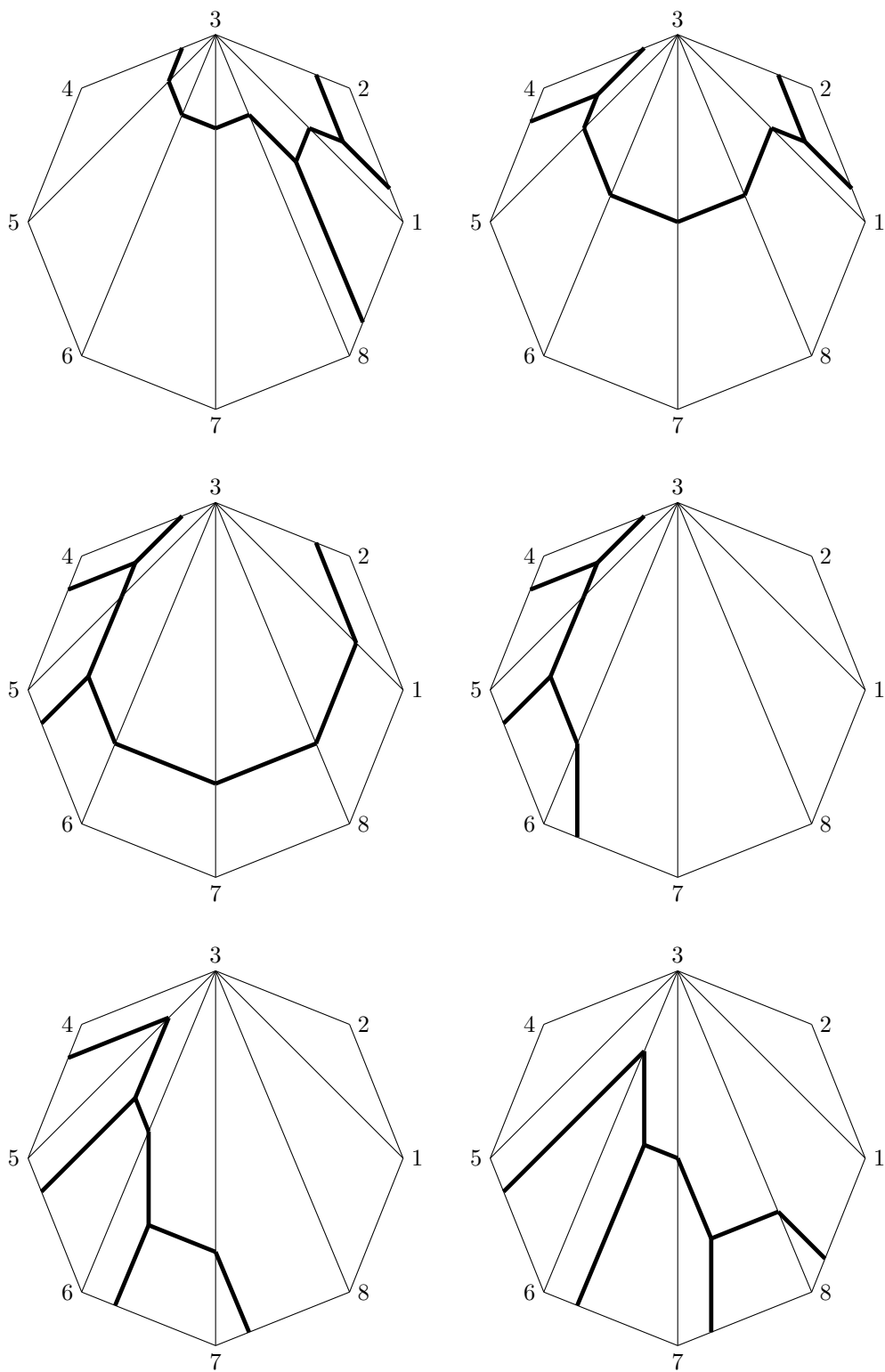
\begin{center}
\begin{tikzpicture}[scale=2]

  \node (3) at (0,1.4) {};
  \node [above] at (3) {$3$};

  \node (4) at (-1,1) {};
  \node [above, left] at (4) {$4$};

  \node (5) at (-1.4,0) {};
  \node [left] at (5) {$5$};

  \node (6) at (-1,-1) {};
  \node [below, left] at (6) {$6$};

  \node (7) at (0,-1.4) {};
  \node [below] at (7) {$7$};

  \node (8) at (1,-1) {};
  \node [below, right] at (8) {$8$};

  \node (1) at (1.4,0) {};
  \node [right] at (1) {$1$};

  \node (2) at (1,1) {};
  \node [above, right] at (2) {$2$};

  \draw [] (0,1.4) to (-1,1);
  \draw [] (-1,1) to (-1.4,0);
  \draw [] (-1.4,0) to (-1,-1);
  \draw [] (-1,-1) to (0,-1.4);
  \draw [] (0,-1.4) to (1,-1);
  \draw [] (1,-1) to (1.4,0);
  \draw [] (1.4,0) to (1,1);
  \draw [] (1,1) to (0,1.4);

  \draw [] (0,1.4) to (0,-1.4);
  \draw [] (0,1.4) to (-1,-1);
  \draw [] (0,1.4) to (-1.4,0);
  \draw [] (0,1.4) to (1,-1);
  \draw [] (0,1.4) to (1.4,0);

  \draw [line width=1.8pt] (0.75,1.1) to (0.95,0.6);
  \draw [line width=1.8pt] (1.3,0.25) to (0.95,0.6);
  \draw [line width=1.8pt] (0.7,0.7) to (0.95,0.6);

  \draw [line width=1.8pt] (0.7,0.7) to (0.6,0.45);
  \draw [line width=1.8pt] (1.1,-0.75) to (0.6,0.45);
  \draw [line width=1.8pt] (0.25,0.8) to (0.6,0.45);

  \draw [line width=1.8pt] (0.25,0.8) to (0,0.7);
  \draw [line width=1.8pt] (-0.25,0.8) to (0,0.7);
  \draw [line width=1.8pt] (-0.25,0.8) to (-0.35,1.05);
  \draw [line width=1.8pt] (-0.25,1.3) to (-0.35,1.05);

\end{tikzpicture} \quad
\begin{tikzpicture}[scale=2]

  \node (3) at (0,1.4) {};
  \node [above] at (3) {$3$};

  \node (4) at (-1,1) {};
  \node [above, left] at (4) {$4$};

  \node (5) at (-1.4,0) {};
  \node [left] at (5) {$5$};

  \node (6) at (-1,-1) {};
  \node [below, left] at (6) {$6$};

  \node (7) at (0,-1.4) {};
  \node [below] at (7) {$7$};

  \node (8) at (1,-1) {};
  \node [below, right] at (8) {$8$};

  \node (1) at (1.4,0) {};
  \node [right] at (1) {$1$};

  \node (2) at (1,1) {};
  \node [above, right] at (2) {$2$};

  \draw [] (0,1.4) to (-1,1);
  \draw [] (-1,1) to (-1.4,0);
  \draw [] (-1.4,0) to (-1,-1);
  \draw [] (-1,-1) to (0,-1.4);
  \draw [] (0,-1.4) to (1,-1);
  \draw [] (1,-1) to (1.4,0);
  \draw [] (1.4,0) to (1,1);
  \draw [] (1,1) to (0,1.4);

  \draw [] (0,1.4) to (0,-1.4);
  \draw [] (0,1.4) to (-1,-1);
  \draw [] (0,1.4) to (-1.4,0);
  \draw [] (0,1.4) to (1,-1);
  \draw [] (0,1.4) to (1.4,0);

  \draw [line width=1.8pt] (0.75,1.1) to (0.95,0.6);
  \draw [line width=1.8pt] (1.3,0.25) to (0.95,0.6);
  \draw [line width=1.8pt] (0.7,0.7) to (0.95,0.6);

  \draw [line width=1.8pt] (0.7,0.7) to (0.5,0.2);
  \draw [line width=1.8pt] (0,0) to (0.5,0.2);
  \draw [line width=1.8pt] (0,0) to (-0.5,0.2);
  \draw [line width=1.8pt] (-0.7,0.7) to (-0.5,0.2);

  \draw [line width=1.8pt] (-0.7,0.7) to (-0.6,0.95);
  \draw [line width=1.8pt] (-1.1,0.75) to (-0.6,0.95);
  \draw [line width=1.8pt] (-0.25,1.3) to (-0.6,0.95);

\end{tikzpicture}
\end{center}

\begin{center}
\begin{tikzpicture}[scale=2]

  \node (3) at (0,1.4) {};
  \node [above] at (3) {$3$};

  \node (4) at (-1,1) {};
  \node [above, left] at (4) {$4$};

  \node (5) at (-1.4,0) {};
  \node [left] at (5) {$5$};

  \node (6) at (-1,-1) {};
  \node [below, left] at (6) {$6$};

  \node (7) at (0,-1.4) {};
  \node [below] at (7) {$7$};

  \node (8) at (1,-1) {};
  \node [below, right] at (8) {$8$};

  \node (1) at (1.4,0) {};
  \node [right] at (1) {$1$};

  \node (2) at (1,1) {};
  \node [above, right] at (2) {$2$};

  \draw [] (0,1.4) to (-1,1);
  \draw [] (-1,1) to (-1.4,0);
  \draw [] (-1.4,0) to (-1,-1);
  \draw [] (-1,-1) to (0,-1.4);
  \draw [] (0,-1.4) to (1,-1);
  \draw [] (1,-1) to (1.4,0);
  \draw [] (1.4,0) to (1,1);
  \draw [] (1,1) to (0,1.4);

  \draw [] (0,1.4) to (0,-1.4);
  \draw [] (0,1.4) to (-1,-1);
  \draw [] (0,1.4) to (-1.4,0);
  \draw [] (0,1.4) to (1,-1);
  \draw [] (0,1.4) to (1.4,0);

  \draw [line width=1.8pt] (0.75,1.1) to (1.05,0.35);
  \draw [line width=1.8pt] (0.75,-0.4) to (1.05,0.35);
  \draw [line width=1.8pt] (0.75,-0.4) to (0,-0.7);
  \draw [line width=1.8pt] (-0.75,-0.4) to (0,-0.7);

  \draw [line width=1.8pt] (-0.75,-0.4) to (-0.95,0.1);
  \draw [line width=1.8pt] (-1.3,-0.25) to (-0.95,0.1);
  \draw [line width=1.8pt] (-0.7,0.7) to (-0.95,0.1);

  \draw [line width=1.8pt] (-0.7,0.7) to (-0.6,0.95);
  \draw [line width=1.8pt] (-1.1,0.75) to (-0.6,0.95);
  \draw [line width=1.8pt] (-0.25,1.3) to (-0.6,0.95);

\end{tikzpicture} \quad
\begin{tikzpicture}[scale=2]

  \node (3) at (0,1.4) {};
  \node [above] at (3) {$3$};

  \node (4) at (-1,1) {};
  \node [above, left] at (4) {$4$};

  \node (5) at (-1.4,0) {};
  \node [left] at (5) {$5$};

  \node (6) at (-1,-1) {};
  \node [below, left] at (6) {$6$};

  \node (7) at (0,-1.4) {};
  \node [below] at (7) {$7$};

  \node (8) at (1,-1) {};
  \node [below, right] at (8) {$8$};

  \node (1) at (1.4,0) {};
  \node [right] at (1) {$1$};

  \node (2) at (1,1) {};
  \node [above, right] at (2) {$2$};

  \draw [] (0,1.4) to (-1,1);
  \draw [] (-1,1) to (-1.4,0);
  \draw [] (-1.4,0) to (-1,-1);
  \draw [] (-1,-1) to (0,-1.4);
  \draw [] (0,-1.4) to (1,-1);
  \draw [] (1,-1) to (1.4,0);
  \draw [] (1.4,0) to (1,1);
  \draw [] (1,1) to (0,1.4);

  \draw [] (0,1.4) to (0,-1.4);
  \draw [] (0,1.4) to (-1,-1);
  \draw [] (0,1.4) to (-1.4,0);
  \draw [] (0,1.4) to (1,-1);
  \draw [] (0,1.4) to (1.4,0);

  \draw [line width=1.8pt] (-0.75,-0.4) to (-0.75,-1.1);

  \draw [line width=1.8pt] (-0.75,-0.4) to (-0.95,0.1);
  \draw [line width=1.8pt] (-1.3,-0.25) to (-0.95,0.1);
  \draw [line width=1.8pt] (-0.7,0.7) to (-0.95,0.1);

  \draw [line width=1.8pt] (-0.7,0.7) to (-0.6,0.95);
  \draw [line width=1.8pt] (-1.1,0.75) to (-0.6,0.95);
  \draw [line width=1.8pt] (-0.25,1.3) to (-0.6,0.95);

\end{tikzpicture}
\end{center}

\begin{center}
\begin{tikzpicture}[scale=2]

  \node (3) at (0,1.4) {};
  \node [above] at (3) {$3$};

  \node (4) at (-1,1) {};
  \node [above, left] at (4) {$4$};

  \node (5) at (-1.4,0) {};
  \node [left] at (5) {$5$};

  \node (6) at (-1,-1) {};
  \node [below, left] at (6) {$6$};

  \node (7) at (0,-1.4) {};
  \node [below] at (7) {$7$};

  \node (8) at (1,-1) {};
  \node [below, right] at (8) {$8$};

  \node (1) at (1.4,0) {};
  \node [right] at (1) {$1$};

  \node (2) at (1,1) {};
  \node [above, right] at (2) {$2$};

  \draw [] (0,1.4) to (-1,1);
  \draw [] (-1,1) to (-1.4,0);
  \draw [] (-1.4,0) to (-1,-1);
  \draw [] (-1,-1) to (0,-1.4);
  \draw [] (0,-1.4) to (1,-1);
  \draw [] (1,-1) to (1.4,0);
  \draw [] (1.4,0) to (1,1);
  \draw [] (1,1) to (0,1.4);

  \draw [] (0,1.4) to (0,-1.4);
  \draw [] (0,1.4) to (-1,-1);
  \draw [] (0,1.4) to (-1.4,0);
  \draw [] (0,1.4) to (1,-1);
  \draw [] (0,1.4) to (1.4,0);

  \draw [line width=1.8pt] (-1.1,0.75) to (-0.35,1.05);

  \draw [line width=1.8pt] (-0.35,1.05) to (-0.6,0.45);
  \draw [line width=1.8pt] (-1.3,-0.25) to (-0.6,0.45);
  \draw [line width=1.8pt] (-0.5,0.2) to (-0.6,0.45);

  \draw [line width=1.8pt] (-0.5,0.2) to (-0.5,-0.5);
  \draw [line width=1.8pt] (-0.75,-1.1) to (-0.5,-0.5);
  \draw [line width=1.8pt] (0,-0.7) to (-0.5,-0.5);

  \draw [line width=1.8pt] (0,-0.7) to (0.25,-1.3);

\end{tikzpicture} \quad
\begin{tikzpicture}[scale=2]

  \node (3) at (0,1.4) {};
  \node [above] at (3) {$3$};

  \node (4) at (-1,1) {};
  \node [above, left] at (4) {$4$};

  \node (5) at (-1.4,0) {};
  \node [left] at (5) {$5$};

  \node (6) at (-1,-1) {};
  \node [below, left] at (6) {$6$};

  \node (7) at (0,-1.4) {};
  \node [below] at (7) {$7$};

  \node (8) at (1,-1) {};
  \node [below, right] at (8) {$8$};

  \node (1) at (1.4,0) {};
  \node [right] at (1) {$1$};

  \node (2) at (1,1) {};
  \node [above, right] at (2) {$2$};

  \draw [] (0,1.4) to (-1,1);
  \draw [] (-1,1) to (-1.4,0);
  \draw [] (-1.4,0) to (-1,-1);
  \draw [] (-1,-1) to (0,-1.4);
  \draw [] (0,-1.4) to (1,-1);
  \draw [] (1,-1) to (1.4,0);
  \draw [] (1.4,0) to (1,1);
  \draw [] (1,1) to (0,1.4);

  \draw [] (0,1.4) to (0,-1.4);
  \draw [] (0,1.4) to (-1,-1);
  \draw [] (0,1.4) to (-1.4,0);
  \draw [] (0,1.4) to (1,-1);
  \draw [] (0,1.4) to (1.4,0);

  \draw [line width=1.8pt] (-1.3,-0.25) to (-0.25,0.8);

  \draw [line width=1.8pt] (-0.25,0.8) to (-0.25,0.1);
  \draw [line width=1.8pt] (-0.75,-1.1) to (-0.25,0.1);
  \draw [line width=1.8pt] (0,0) to (-0.25,0.1);

  \draw [line width=1.8pt] (0,0) to (0.25,-0.6);
  \draw [line width=1.8pt] (0.25,-1.3) to (0.25,-0.6);
  \draw [line width=1.8pt] (0.75,-0.4) to (0.25,-0.6);

  \draw [line width=1.8pt] (0.75,-0.4) to (1.1,-0.75);

\end{tikzpicture}
\end{center}

\begin{center}
\begin{tikzpicture}[scale=2]

  \node (3) at (0,1.4) {};
  \node [above] at (3) {$3$};

  \node (4) at (-1,1) {};
  \node [above, left] at (4) {$4$};

  \node (5) at (-1.4,0) {};
  \node [left] at (5) {$5$};

  \node (6) at (-1,-1) {};
  \node [below, left] at (6) {$6$};

  \node (7) at (0,-1.4) {};
  \node [below] at (7) {$7$};

  \node (8) at (1,-1) {};
  \node [below, right] at (8) {$8$};

  \node (1) at (1.4,0) {};
  \node [right] at (1) {$1$};

  \node (2) at (1,1) {};
  \node [above, right] at (2) {$2$};

  \draw [] (0,1.4) to (-1,1);
  \draw [] (-1,1) to (-1.4,0);
  \draw [] (-1.4,0) to (-1,-1);
  \draw [] (-1,-1) to (0,-1.4);
  \draw [] (0,-1.4) to (1,-1);
  \draw [] (1,-1) to (1.4,0);
  \draw [] (1.4,0) to (1,1);
  \draw [] (1,1) to (0,1.4);

  \draw [] (0,1.4) to (0,-1.4);
  \draw [] (0,1.4) to (-1,-1);
  \draw [] (0,1.4) to (-1.4,0);
  \draw [] (0,1.4) to (1,-1);
  \draw [] (0,1.4) to (1.4,0);

  \draw [line width=1.8pt] (-0.75,-1.1) to (0,0.7);

  \draw [line width=1.8pt] (0,0.7) to (0.25,0.1);
  \draw [line width=1.8pt] (0.25,-1.3) to (0.25,0.1);
  \draw [line width=1.8pt] (0.5,0.2) to (0.25,0.1);

  \draw [line width=1.8pt] (0.5,0.2) to (0.85,-0.15);
  \draw [line width=1.8pt] (1.1,-0.75) to (0.85,-0.15);
  \draw [line width=1.8pt] (1.05,0.35) to (0.85,-0.15);

  \draw [line width=1.8pt] (1.3,0.25) to (1.05,0.35);

\end{tikzpicture} \quad
\begin{tikzpicture}[scale=2]

  \node (3) at (0,1.4) {};
  \node [above] at (3) {$3$};

  \node (4) at (-1,1) {};
  \node [above, left] at (4) {$4$};

  \node (5) at (-1.4,0) {};
  \node [left] at (5) {$5$};

  \node (6) at (-1,-1) {};
  \node [below, left] at (6) {$6$};

  \node (7) at (0,-1.4) {};
  \node [below] at (7) {$7$};

  \node (8) at (1,-1) {};
  \node [below, right] at (8) {$8$};

  \node (1) at (1.4,0) {};
  \node [right] at (1) {$1$};

  \node (2) at (1,1) {};
  \node [above, right] at (2) {$2$};

  \draw [] (0,1.4) to (-1,1);
  \draw [] (-1,1) to (-1.4,0);
  \draw [] (-1.4,0) to (-1,-1);
  \draw [] (-1,-1) to (0,-1.4);
  \draw [] (0,-1.4) to (1,-1);
  \draw [] (1,-1) to (1.4,0);
  \draw [] (1.4,0) to (1,1);
  \draw [] (1,1) to (0,1.4);

  \draw [] (0,1.4) to (0,-1.4);
  \draw [] (0,1.4) to (-1,-1);
  \draw [] (0,1.4) to (-1.4,0);
  \draw [] (0,1.4) to (1,-1);
  \draw [] (0,1.4) to (1.4,0);

  \draw [line width=1.8pt] (0.75,1.1) to (0.95,0.6);
  \draw [line width=1.8pt] (1.3,0.25) to (0.95,0.6);
  \draw [line width=1.8pt] (0.7,0.7) to (0.95,0.6);

  \draw [line width=1.8pt] (0.7,0.7) to (0.6,0.45);
  \draw [line width=1.8pt] (1.1,-0.75) to (0.6,0.45);
  \draw [line width=1.8pt] (0.25,0.8) to (0.6,0.45);

  \draw [line width=1.8pt] (0.25,0.8) to (0.25,-1.3);

\end{tikzpicture}
\end{center}

\begin{figure}[h]
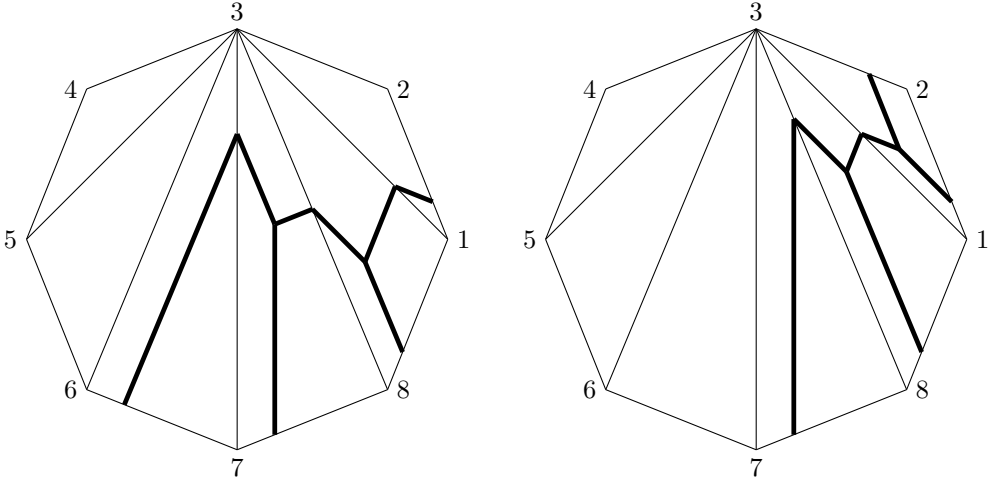


\caption{The convex hulls of the tropical points $l_1,\dots,l_8$ \label{fig:lsconvexhulls}}
\end{figure}

From this analysis, we can state our key observation.

\begin{observation} Take a point $x^t \in \Conf^+_8 \A'(\Qt)$ corresponding to a higher lamination $\{x_i\}$ such that $d(x_{i-1}, x_i) = a_i \omega_1$. The convex hull of the coordinates of $x^t$ will be as depicted in Figure~\ref{fig:xtyt} for each triangle in the triangulation of the octagon. Imagine gluing these six triangles together to re-form the octagon. Then the segments occurring in the this octagon are exactly the union of the bold segments in Figure~\ref{fig:lsconvexhulls}.
\end{observation}

We can now prove Proposition~\ref{uniquelift}.

\begin{proof}[Proof of Proposition~\ref{uniquelift} and Theorem~\ref{distinguished}]
Recall that in Figures \ref{fig:viewfromabove},\ref{fig:oneside}, \ref{fig:hulltwosides},\ref{fig:xtyt}, and \ref{fig:lsconvexhulls}, a segment in the interior of any triangle corresponds to one of the hive inequalities. In Figure~\ref{fig:lsconvexhulls}, the bold segments are associated to hive inequalities that hold strictly.

Suppose we start with any $y^t \in \Conf_8 V'(\Qt)$. Form its lift $x^t$. The segments in Figure~\ref{fig:xtyt} correspond to all the hive inequalities that need to hold in order for $x^t$ to satisfy the hive inequalities. The places where segments are missing are places where we already have equality for the hive inequalities for $x^t$' Using the above observation, we see that by adding a suitable positive linear combination of the $l_i$ to $x^t$, we get a point $x^t + \sum c_il_i$ that satisfies the hive inequalities. Moreover,  $x^t + \sum c_il_i$ is the lift of $y^t = \pi(\sum c_il_i)$. Thus our previous analysis tells us that Equation~\eqref{keyeqn} holds for $x^t + \sum c_il_i$. Note that Equation~\eqref{keyeqn} clearly holds for $\sum c_il_i$, so by linearity, Equation~\eqref{keyeqn} holds for $x^t$. This proves the existence of the lift of $y^t$. Thus we have a natural map
$$s: \Conf_8 V'(\Qt) \hookrightarrow \Conf_8 \A'(\Qt)$$
mapping $y^t$ to the unique $x^t$ such that $\pi(x^t)=y^t$ and $x^t$ satisfies Equation~\eqref{keyeqn}. This proves Proposition~\ref{uniquelift}. This means that among all $x^t \in \Conf_8 \A'(\Qt)$ such that $\pi(x^t)=y^t$, there is a distinguished representative satisfying Equation~\eqref{keyeqn}.

All that remains is to show that the inequalities defining $\Conf^+_n V'(\Qt) \subset \Conf_n V'(\Qt)$ are precisely Equation~\eqref{ineq}.

For any $y^t \in \in \Conf_8 V'(\Qt)$, $y^t$ is represented by a higher lamination $\{x_i\}$ if and only if $s(y^t)$ satisfies the hive inequalities. Now, in \cite{GS} it is shown that the hive inequalities are satisfied in an $n$-gon if and only they are satisfied in all the triangles involving consecutive vertices. But the analysis in Figure~\ref{fig:hulltwosides} showed that the hive inequalities on such triangles are given exactly by Equation~\eqref{ineq}. This completes the proof of our main theorem, Theorem~\ref{distinguished}.
\end{proof}

\section{Duality conjectures}\label{secn:duality}

The philosophy of \cite{FG2} tells us that cluster varieties come in pairs, as part of a cluster ensemble. A cluster ensemble consists of a pair of spaces, an $\A$-space and an $\X$-space. The duality conjectures state, roughly, that integral tropical points of one space parameterize a canonical basis of functions on the other space. These conjectures have been proven in many cases. Here we deal with a version of the conjectures that involves the $\A$-space on both sides. The statement results from the duality conjectures plus some additional analysis of frozen variables.

Here are the duality conjectures for the space $\Conf_n \A$.

\begin{theorem}\label{thm:duality} \cite{GS} The set $\Conf^+_n \A'(\Z^t)$ (cf.~\eqref{eq:hiveinequalities}) parameterizes a basis of functions in $\cO(\Conf_n \A)$. More specifically, suppose that  $x^t \in \Conf^+_n \A'(\Z^t)$ is given by a configuration $(x_1, x_2, \dots, x_n)$ in the building. Let $d(x_{i-1}, x_i) = \lambda_i$, where indices are taken cyclically$\mod n$. Then $x^t $ parameterizes a function that lies in the invariant space 
$$[V_{\lambda_1}^* \otimes V_{\lambda_2}^* \otimes \cdots \otimes V_{\lambda_n}^*]^G.$$
\end{theorem}

Let us rephrase this theorem. Define 
$$\Conf^+_n \A'(\Z^t)(\lambda_1, \dots, \lambda_n) \subset \Conf^+_n \A'(\Z^t)$$ to be the set of higher laminations 
$$(x_1, x_2, \dots, x_n)$$ such that $d(x_{i-1}, x_i) = \lambda_i$. Then an equivalent form of the conjecture is that $$\Conf^+_n \A'(\Z^t)(\lambda_1, \dots, \lambda_n)$$ parameterizes a basis of functions in 
$$[V_{\lambda_1}^* \otimes V_{\lambda_2}^* \otimes \cdots \otimes V_{\lambda_n}^*]^G,$$
as the space of functions $\cO(\Conf_n \A)$ is naturally a direct sum over all such invariant spaces.

We have an analogous statement for the space $\Conf_n V$. Let $a_1, \dots, a_n$ be positive integers such that $\sum a_i$ is a multiple of $k$. Then Equation~\eqref{frozens} give us values for the frozen variables $F_i$. Define 

$$\Conf^+_n V'(\Zt)(a_1, \dots, a_n) \subset \Conf^+_n V'(\Zt)$$
to be the set of tropical points with these prescribed values for $F_i$. Then $y^t \in \Conf^+_n V'(\Zt)(a_1, \dots, a_n)$ corresponds to a higher lamination 
$$(x_1, x_2, \dots, x_n)$$ such that $d(x_{i-1}, x_i) = a_i \omega_1$. (Note that we are using integrality of the $a_i$ to conclude integrality of $s(y^t)$ using Equation \eqref{keyeqn}. Also note that the coordinates for hives are coordinates on $\Conf^+_n \A$, so they will have denominators of $\frac{1}{k}$ for integral points of $\Conf^+_n \A'$.) By Theorem~\ref{thm:duality}, this corresponds to an invariant in the space
$$[V_{a_1\omega_1}^* \otimes V_{a_2\omega_1}^* \otimes \cdots \otimes V_{a_n \omega_1}^*]^G.$$
Now the functions in $\cO(\Conf_n V)$ are a direct sum over all such invariant spaces.

\begin{rmk} Not all integral points of $\Conf^+_n V'(\Zt)$ will occur as $\Conf^+_n V'(\Zt)(a_1, \dots, a_n)$ for some integer $a_i$. This is because inverting Equation~\eqref{frozens} introduces denominators. 
\end{rmk}

There is an integral structure on $\Conf^+_n V'(\Qt)$ such that its integral points are precisely those that occur as $\Conf^+_n V'(\Zt)(a_1, \dots, a_n)$ for some integer $a_i$. Call this space $\Conf^{*,+}_n V'$. Then we have

\begin{theorem}
$\Conf^{*,+}_n V' (\Zt)$ parameterizes a basis of functions in $\cO(\Conf_n V)$. If $y^t \in \Conf^+_n V'(\Zt)(a_1, \dots, a_n)$, then $y^t$ corresponds to a higher lamination 
$$(x_1, x_2, \dots, x_n)$$ such that $d(x_{i-1}, x_i) = a_i \omega_1$. Moreover, $y^t$ parameterizes a function in the invariant space
$$[V_{a_1\omega_1}^* \otimes V_{a_2\omega_1}^* \otimes \cdots \otimes V_{a_n \omega_1}^*]^G.$$
\end{theorem}

Finally, let us compare our results to those in \cite{RW}. In that paper, the authors show that a certain polytope $NO_G^r$, given by equalities similar to Equation~\eqref{ineq}, parameterizes bases in the spaces of invariants
$$[V_{a_1\omega_1}^* \otimes V_{a_2\omega_1}^* \otimes \cdots \otimes V_{a_n \omega_1}^*]^G$$
where $\sum a_i = r$. Thus, one can view our results as a refinement of the corresponding statement in \cite{RW} (Theorem 5.9).

Let us elaborate on this. The condition $\sum a_i = r$ defines a linear subspace of $\Conf_n V' (\Rt)$. The intersection of this subspace with the cone $\Conf^{+}_n V' (\Rt)$ gives the polytope $NO_G^r$. The integral points of this polytope are the intersection of the subspace $\sum a_i = r$ with $\Conf^{*,+}_n V' (\Zt)$. If we additionally fix the $a_i$, we get smaller dimensional subspaces. We can intersect these subspaces with $\Conf^{+}_n V' (\Rt)$ to get a polytope whose integral points are $\Conf^+_n V'(\Zt)(a_1, \dots, a_n)$. These polytopes are slices of the polytope $NO_G^r$.

It should be possible to explicitly derive the inequalities in \cite{RW} from our inequalities in Equation~\eqref{ineq}. The tropical points in \cite{RW} are slightly different than ours, so the literal polytopes $\{\sum a_i = r\} \cap \Conf_n V' (\Rt)$ and $NO_G^r$ differ by a linear transformation. However, the $\X$-coordinates (defined in the following section) should be the same for all the points in the polytopes, as corresponding points differ by a transformation the action of the lineality space.

We propose the following bijection. We will need to make a construction dual to that in the previous Section 6. The constructions differ by a DT-transformation, in a sense. There is another section, 
$$s^\dagger: \Conf_n V'(\Qt) \hookrightarrow \Conf_n \A'(\Qt)$$
defined so that it maps a tropical point $y^t \in \Conf^{\dagger,+}_n V'(\Qt)$ to $x^t \in \Conf^+_n \A'(\Qt)$ corresponding to a higher lamination $\{x_i\}$ such that 
$d(x_{i+1}, x_{i}) = a_i \omega_1$. Then we can define $\Conf^{\dagger,+}_n V'(\Zt)(a_1, \dots, a_n)$ as before.

\begin{conj} Let $y^t \in \Conf^{\dagger,+}_n V'(\Zt)(a_1, \dots, a_n)$. Then act on $y^t$ by the element
$$\frac{1}{k}(b_1, \dots, b_n) \in T^n(\Zt)$$
where $$b_i = \sum_{j < i} a_j.$$
The result will be the corresponding point in $NO_G^r$.
\end{conj}

\section{The \texorpdfstring{$\X$}{X}-space and its tropicalization}

We now describe the $\X$-variety associated to $\tGr(k,n) = \Conf_n V$. In particular we justify our previous claim that the $\X$-variety is $\Conf_n \mathbb{P}(V) = \Conf_n \mathbb{P}^{k-1}$. This turns out to be fairly straightforward.

\begin{theorem} $\Conf_n \mathbb{P}^{k-1}$ has the structure of a cluster $\X$-variety. This $\X$-variety, together with the $\A$-variety $\Conf_n V$, forms a cluster ensemble.
\end{theorem}

Suppose we have a cluster $\A$-variety with seed $\Sigma=(I, I_0,B, d)$. Then for every non-frozen index $i \in I$, there is a cluster variable $X_i$ on the $\X$-variety. There is a map of algebraic torii $p: \A_{\Sigma} \to \X_{\Sigma}$ given by 
$$p^*(X_i) = \prod_{j \in I}A_j^{B_{ij}}.$$

The functions $p^*(X_i)$ a priori live on $\Conf_n V$. However, it is easy to check that they are invariant under the action of $T^n$. We can check this in any one cluster (for example, the one in Figure 1 for $\tGr(4,8)$), and then it will automatically hold for any other cluster. A dimension count shows that $T^n$ is precisely the kernel of the map $p$.

Thus the functions $X_i$ can be viewed as functions on $\Conf_n V / T^n = \Conf_n \mathbb{P}^{k-1}$ and $\X_{\Sigma}$ must be dense in $\Conf_n \mathbb{P}^{k-1}$.

Let $\Conf_n \mathbb{P}^{k-1}(\Zt)$ be the set of integral positive tropical points of $\Conf_n \mathbb{P}^{k-1}$. We have that

$$\Conf_n \mathbb{P}^{k-1}(\Zt) = \Conf_n V(\Zt) / T^n(\Zt).$$ 

This gives us the following statement:

\begin{theorem}
The space $\Conf_n \mathbb{P}^{k-1}(\Zt)$ parameterizes the lineality spaces of $\Conf_n V(\Zt)$.
\end{theorem}

The lineality spaces for horocycle laminations were analyzed in Section 5.1, while the lineality spaces for distinguished representatives were analyzed in Corollary \ref{lineality}.

\section*{Appendix: Cluster algebras}

We review here the basic definitions of cluster algebras. Cluster algebras are commutative algebras that come equipped with a collection of distinguished set of algebra generators, called \emph{cluster variables} or $\A$-coordinates. These generators are grouped in specific subsets known as {\it clusters}. A given cluster does not generate the algebra, but each cluster is a transcendence basis for the field of fractions. There is a procedure known as {\it mutation} which replaces a given cluster with a new one.

Each cluster belongs to a {\it seed}, which roughly consists of the cluster together with the extra underlying data of a $B$-matrix (see below). The $B$-matrix prescribes how to mutate from one seed to any adjacent seed. Starting from an initial choice of seed, all other seeds (and hence, all clusters, and all cluster variables) can be obtained recursively by applying a sequence of mutations. Because mutation is involutive, the resulting pattern of clusters and cluster variables does not depend on the choice of initial seed. 

Each cluster provides a coordinate system on the $\A$-space. The same combinatorial data underlying a seed gives rise to a second, related, algebra of functions, the algebra of $\X$-coordinates. Just as before, the $\X$-coordinates come in specified groupings that mutate when passing from one seed to another. The $\X$-coordinates are functions on the $\X$ space. The $\A$-coordinates and $\X$-coordinates are related by a canonical monomial transformation, which gives a map from the $\A$-space to the $\X$-space. Together, the data of the $\A$-space and the $\X$-space, along with their distinguished sets of coordinates, is called a {\it cluster ensemble}.

Now we give more precise definitions. A seed $\Sigma = (I,I_0,B,d)$ consists of the following data: 
\begin{enumerate}
\item An index set $I$ with a subset $I_0 \subset I$ of ``frozen'' indices. 
\item A rational $I \times I$ \emph{exchange matrix} $B$. It should have the property that $b_{ij} \in \Z$ unless both $i$ and $j$ are frozen.  
\item A set $d = \{d_i \}_{i \in I}$ of positive integers that skew-symmetrize $B$; that is, $b_{ij}d_j = -b_{ji}d_i$ for all $i,j \in I$.
\end{enumerate}

For most purposes, the values of $d_i$ are only important up to simultaneous scaling. Also note that the values of $b_{ij}$ where $i$ and $j$ are both frozen will play no role in the cluster algebra, though it is sometimes convenient to assign values to $b_{ij}$ for bookkeeping purposes.

In the simplest case, $B$ is skew-symmetric, i.e. each $d_i = 1$. In this case, the seed is given simply by the data $\Sigma = (I,I_0,B)$. Moreover, we can depict the $B$-matrix by a {\it quiver} (as was done in previous sections). This quiver will have vertices labelled by the set $I$. When $b_{ij} > 0$, we will have $b_{ij}$ arrows going from $j$ to $i$.

Let $k \in I \setminus I_0$ be an unfrozen index of a seed $\Sigma$.  We say another seed $\Sigma' = \mu_k(\Sigma)$ is obtained from $\Sigma$ by mutation at $k$ if we identify the index sets in such a way that the frozen variables are preserved, and the exchange matrix $B'$ of $\Sigma'$ satisfies
\begin{align}\label{eq:matmut}
b'_{ij} = \begin{cases}
-b_{ij} & i = k \text{ or } j=k \\
b_{ij} & b_{ik}b_{kj} \leq 0 \\
b_{ij} + |b_{ik}|b_{kj} & b_{ik}b_{kj} > 0.
\end{cases}
\end{align}
We can view the procedure of mutation of a seed as follows. We consider the arrows involving the vertex $k$. For every pair of arrows such that one arrow goes into $k$ and one arrow goes out of $k$, we compose these arrows to get a new arrow. Then we cancel arrows in opposite directions. Finally, we reverse all the arrows going in or out of $k$. Two seeds $\Sigma$ and $\Sigma'$ are said to be mutation equivalent if they are related by a finite sequence of mutations.

To a seed $\Sigma$ we associate a collection of \emph{cluster variables} $\{A_i\}_{i \in I}$ and a split algebraic torus $\A_\Sigma := \Spec \Z[A^{\pm 1}_I]$, where $\Z[A^{\pm 1}_I]$ denotes the ring of Laurent polynomials in the cluster variables.  If $\Sigma'$ is obtained from $\Sigma$ by mutation at $k \in I \setminus I_0$, there is a birational \emph{cluster transformation} $\mu_k: \A_\Sigma \to \A_{\Sigma'}$.  This is defined by the \emph{exchange relation}

\begin{align}\label{eq:Atrans}
\mu_k^*(A'_i) = \begin{cases}
A_i & i \neq k \\
A_k^{-1}\biggl(\prod_{b_{kj}>0}A_j^{b_{kj}} + \prod_{b_{kj}<0}A_j^{-b_{kj}}\biggr) & i = k.
\end{cases}
\end{align}
Composing these transformations yields gluing data between any tori $\A_\Sigma$ and $\A_{\Sigma'}$ of mutation equivalent seeds $\Sigma$ and $\Sigma'$.  The $\A$-space $\A_{|\Sigma|}$ is defined as the scheme obtained from gluing together all such tori of seeds mutation equivalent with an initial seed $\Sigma$.  

Given a seed $\Sigma$ we also associate a second algebraic torus $\X_\Sigma := \Spec \Z[X_I^{\pm 1}]$, where $\Z[X_I^{\pm 1}]$ again denotes the Laurent polynomial ring in the variables $\{X_i\}_{i \in I}$.  If $\Sigma'$ is obtained from $\Sigma$ by mutation at $k \in I \setminus I_0$, we again have a birational map $\mu_k: \X_\Sigma \to \X_{\Sigma'}$.  It is defined by
\begin{align}\label{eq:Xtrans}
\mu_k^*(X'_i) = \begin{cases}
X_iX_k^{[b_{ik}]_+}(1+X_k)^{-b_{ik}} & i \neq k  \\
X_k^{-1} & i = k,
\end{cases}
\end{align}
where $[b_{ik}]_+:=\mathrm{max}(0,b_{ik})$.  The $\X$-space $\X_{|\Sigma|}$ is defined as the scheme obtained from gluing together all such tori of seeds mutation equivalent with an initial seed $\Sigma$.

Now we will describe the natural map from $\A_{\Sigma}$ to $\X_{\Sigma}$. Let us assume that the entries of the $B$-matrix are all integers. Then we can define $p: \A_\Sigma \to \X_\Sigma$ by
$$p^*(X_i) = \prod_{j \in I}A_j^{B_{ij}}.$$

This formula appears to depend on the seed, but it actually intertwines the mutation of both the $\A$-coordinates and the $\X$-coordinates. In other words, if $\Sigma'$ is obtained from $\Sigma$ by mutation at $k$, there is a commutative diagram

\vspace{-2mm}
 \[
\begin{tikzcd}
\A_{\Sigma} \arrow[dashed]{r}{\mu_k} \arrow{d}{p} & \A_{\Sigma'} \arrow{d}{p'} \\
\X_{\Sigma} \arrow[dashed]{r}{\mu_k} & \X_{\Sigma'}.
\end{tikzcd}
\]
So the maps $\A_{\Sigma}$ to $\X_{\Sigma}$ glue to give a map $\A_{|\Sigma|}$ to $\X_{|\Sigma|}$.

\end{document}